\documentclass[11pt]{amsart}

\usepackage[T1]{fontenc}
\usepackage{amsmath,amssymb,amsfonts,amsthm,mathtools,mathrsfs}
\usepackage{enumitem}
\usepackage[margin=1in]{geometry}
\usepackage[colorlinks=true,linkcolor=blue,citecolor=blue,urlcolor=blue]{hyperref}

\numberwithin{equation}{section}
\allowdisplaybreaks
\theoremstyle{plain}
\newtheorem{theorem}{Theorem}[section]
\newtheorem{lemma}[theorem]{Lemma}
\newtheorem{proposition}[theorem]{Proposition}
\newtheorem{corollary}[theorem]{Corollary}

\theoremstyle{definition}
\newtheorem{definition}[theorem]{Definition}
\newtheorem{aconjecture}{Conjecture}

\theoremstyle{remark}
\newtheorem{remark}[theorem]{Remark}

\newcommand{\Hom}{\mathrm{Hom}}
\newcommand{\End}{\mathrm{End}}
\newcommand{\GL}{\mathrm{GL}}
\newcommand{\Cl}{\mathrm{Cl}}
\newcommand{\Aut}{\mathrm{Aut}}

\newcommand{\Q}{\mathbb{Q}}
\newcommand{\Z}{\mathbb{Z}}

\newcommand{\CO}{\mathcal{O}}

\begin{document}

\title[Yun and overorder zeta functions]{Yun's zeta function and the overorder zeta function for Gorenstein cubic orders}

\author{Taiwang Deng}
\address[T.D.]{Beijing Institute of Mathematical Sciences and Applications (BIMSA), Huairou District, 100084, Beijing\\
China}
\email{dengtaiw@bimsa.cn}

\author{Malors Espinosa}
\address[M.E.]{Yau Mathematical Sciences Center, Tsinghua University, Haidian District, Beijing, 100084, China.}
\email{espino41@mail.tsinghua.edu.cn}


\date{}

\begin{abstract}
For every Gorenstein cubic \(\Z\)-order, we prove an identity equating Yun's zeta function,
defined by counting finite-index submodules of the trace dual, with the
explicit overorder zeta function introduced in our previous work on Beyond Endoscopy
for \(\mathrm{GL}_3\). This is posed as Conjecture A in an early draft of Deng-Espinosa and  Lee subsequently
proved the functional equation of the overorder
zeta function, making the Deng--Espinosa isolation of the trivial representation
fully unconditional. His argument computes the local factors explicitly.  
Our proof is independent of Lee's and does not evaluate the individual
cubic overorder factors: it matches natural decompositions of the two sides
and concludes by induction.  As an application, we give a short, uniform
evaluation of the local \(\mathrm{GL}_3\) Kloosterman Dirichlet series
in the Poisson-summation argument of Deng--Espinosa.  The direct local
 analysis of the Kloosterman series occupies nearly ninety pages in Deng--Espinosa while our treatment  here replaces its case-by-case enumeration with a
short uniform proof.  
\end{abstract}

\maketitle

\tableofcontents

\section{Introduction}\label{sec:introduction}

\subsection{Background}
Zeta functions of arithmetic orders were introduced by Solomon
\cite{Solomon77} and developed systematically by Bushnell and Reiner
\cite{BushnellReiner80,BushnellReiner81a,BushnellReiner81b}.  For
nonmaximal orders, these functions interpolate between the arithmetic of
the order and that of its normalization.  If \(R\) is the maximal order of
a number field, its trace dual is invertible and its ideal-counting series
is the usual Dedekind zeta function.  For a nonmaximal order the trace dual
need not be invertible, and one must count all of its finite-index
\(R\)-submodules.  Yun introduced the resulting series \(J_R(s)\) and proved
that its completion has the same basic analytic properties as a Dedekind
zeta function \cite[Theorem~1.2]{ZYun}.  Locally, the normalized factor is a
polynomial in \(q^{-s}\) and satisfies a functional equation
\cite[Theorem~2.5]{ZYun}.  When the order is generated by a regular
semisimple element, its special values recover the corresponding orbital
integral \cite[Corollary~4.6]{ZYun}.  Yun's construction therefore links the
ideal theory of an order with the orbital integrals attached to its
generators.

This link is particularly important in the Beyond Endoscopy program,
proposed by Langlands as a trace-formula approach to functoriality
\cite{LanBE04}.  Altu\u{g}'s work for \(\GL_2\) showed how Poisson summation
can isolate special spectral contributions from the regular elliptic part of
the trace formula \cite{AliI}.  Arthur \cite{Arthur} identified a principal arithmetic
obstacle to extending this method to \(\GL_{n}\): one needs a higher-rank
expression for the finite orbital integrals that generalizes Langlands'
divisor-sum formula and retains the symmetry required by Poisson summation.
For the monogenic order
\(R_\gamma=\Z[\lambda]/(p_\gamma(\lambda))\), Arthur emphasized Yun's
identity
\[
\operatorname{Orb}(f_q,\gamma)
=\widetilde J_{\gamma,q}(0)
=\widetilde J_{\gamma,q}(1),
\]
where \(f_q\) is the characteristic function of the space of matrices 
\(M_{n}(\Z_q)\).  He interpreted Yun's polynomial as a natural
extension of the orbital integral, whose functional equation supplies
the symmetry \(s\mapsto1-s\).  For \(\GL_2\), this polynomial specializes to
Langlands' local orbital-integral formula \cite[Lemma~1]{LanBE04}, used by
Altu\u{g} in \cite[\S2.2.2]{AliI} through the disguise of Zagier's L function \cite{Zagier}.  Arthur proposed that the same
structure should enter Poisson summation in higher rank
\cite[\S3]{Arthur}.

\subsection{Main results}
In our work \cite{DengEspinosaGL3PS} (Deng-Espinosa),  for \(\GL_3\) we isolate the contribution of the trivial representation to the elliptic part 
of Arthur Selberg trace formula conditional on a conjecture 
(Conjecture~A) which predicts that, for a Gorenstein quadratic or
cubic order \(R\), Yun's zeta function agrees with an overorder zeta
function introduced there after restoring the Riemann zeta factor:
\[
J_R(s)=L(s,R)\zeta_\Q(s).
\]
The main result of this paper proves the local version of this identity for every Gorenstein cubic order and every residue
characteristic.

\begin{theorem}[Main theorem]
\label{thm:local-conjecture-a}
Let \(q\) be a prime, let \(E\) be an \'etale cubic \(\Q_q\)-algebra, and let
\(A\subset E\) be a Gorenstein cubic \(\Z_q\)-order.  With the local
normalizations of Section~\ref{sec:local-normalization} and
\(t=q^{-s}\), one has
\begin{equation}\label{eq:local-comparison-final}
J_A(s)=\widetilde L_A(t).
\end{equation}
Equivalently,
\[
\widetilde J_A(s)=\widetilde L_q(s,A).
\]
\end{theorem}

Theorem~\ref{thm:local-conjecture-a} is proved in
Section~\ref{sec:proof-local-comparison}.  Applying it at every prime gives Corollary~\ref{cor:global-conjecture-a}: Conjecture~A holds for every
Gorenstein cubic order over \(\Z\).

The theorem identifies Yun's intrinsic lattice-counting series with the
explicit local factor whose coefficients occur in the Deng--Espinosa
Poisson-summation argument.  It therefore confirms, in the first
higher-rank case, Arthur's expectation for Yun's
zeta function.  

Combining the above theorem with the parametrization of ideals by
Hermite normal form (HNF) yields a second main result.  Given a prime \(q\) and
a fixed \(0\ne c\in\Z_q\), let \(D_q(z)\) denotes the
corresponding local Kloosterman series defined in
Section~\ref{sec:kloosterman-application}.

\begin{theorem}
\label{thm:kloosterman-local-factor}
 If
\(e=v_q(c)\) and \(x=q^{-z}\), then
\begin{equation}\label{eq:kloosterman-local-factor-rational}
D_q(z)
=
\frac{(1-x^{e+1})(1-x^{e+2})
(1-q^{-1}x)(1-q^{-1}x^2)}
{(1-x)(1-x^2)^2(1-x^3)}.
\end{equation}
In terms of \(z\), we have 
\[
D_q(z)
=
\frac{
(1-q^{-z(e+1)})(1-q^{-z(e+2)})
(1-q^{-z-1})(1-q^{-2z-1})
}{
(1-q^{-z})(1-q^{-2z})^2(1-q^{-3z}).
}.
\]
\end{theorem}

Theorem~\ref{thm:kloosterman-local-factor} is proved in
Section~\ref{sec:kloosterman-application}.  Its proof is
uniform in \(q\), including residue characteristics \(2\) and \(3\), and
treats unit and ramified constant terms simultaneously. 
The same  argument can also be applied to shorten and unify 
the auxiliary local-factor calculation in Altu\u{g}'s work for
\(\mathrm{GL}_2\).

 If \(p\) is a
prime and \(k\ge1\), take the global parameter \(c=\pm p^k\), so that the
polynomial has constant coefficient \(\mp p^k\).   We then obtain, for \(\Re(z)>1\),
\begin{equation}\label{eq:introduction-kloosterman-euler-product}
\prod_q D_q(z)
=
\frac{\zeta(2z)\zeta(3z)}
{\zeta(z+1)\zeta(2z+1)}
\frac{1-p^{-z(k+1)}}{1-p^{-z}}
\frac{1-p^{-z(k+2)}}{1-p^{-2z}}.
\end{equation}
 This recovers the formula occurring in
Deng--Espinosa.

Although we keep our base field to be $\Q$ (resp. $\Q_q$) to keep the same notation with \cite{DengEspinosaGL3PS},
the above theorems naturally extend to any number field $F$ (resp. finite extension of $\Q_q$).

\subsection{Previous work}
Deng--Espinosa develop Arthur's proposed mechanism for \(\GL_3(\Q)\)
\cite[\S1]{DengEspinosaGL3PS}.  Their immediate objective is a concrete
expansion of
\[
\mathrm I_{\mathrm{ell}}(f)-\operatorname{Tr}(\mathbf 1(f)),
\]
the difference between the regular elliptic contribution and the trace of
the trivial representation.  Isolating the trivial representation is
essential because it contributes a non-tempered term that must be
removed before the remaining spectral average can converge
\cite{frenkel2010formule}.
Regular elliptic classes are described by cubic characteristic polynomials
and hence by the monogenic cubic orders generated by their roots. Following Altu\u{g} \cite{AliI},
Deng--Espinosa rewrite the finite orbital factors as values of zeta
functions of cubic orders, apply an approximate functional equation, and
then use Poisson summation in the coefficients of the characteristic
polynomial.  The functional equation produces the two weighted Dirichlet
sums from which Poisson summation and contour shifting recover the trivial
representation as a residue
\cite[\S5, \S7, and~\S 9]{DengEspinosaGL3PS}.

For this purpose, Deng--Espinosa introduce an overorder zeta function
\(L(s,R)\).  It is a finite sum over the orders between \(R\) and its
normalization, with a correction factor determined by the Cohen--Macaulay
type of each local cubic overorder
\cite[Definition~12]{DengEspinosaGL3PS}.  Its coefficients are suited to
Poisson summation, and its value at \(s=1\) gives the required orbital
integral \cite[Proposition~19]{DengEspinosaGL3PS}.  Deng--Espinosa originally
obtained its functional equation by conjecturing its relation to Yun's zeta function.

Lee subsequently computed all local cubic overorder factors explicitly
\cite[Theorems~4.3, 4.5, 4.7, 4.9, and~4.13]{LeeCubicFE} and proved their
local functional equations
\cite[Theorems~5.3 and~5.4]{LeeCubicFE}.  His global functional equation
\cite[Theorem~1.3]{LeeCubicFE} applies to Gorenstein cubic orders.  As a consequence, he made the Deng--Espinosa isolation of the
trivial representation fully unconditional.  The present paper addresses
a different question: it proves the conjectural identification
of the two zeta functions, without computing the complete
local factor explicitly.

As an application, we obtain a new proof of the evaluation of the Dirichlet series associated with the Kloosterman-type sum of Deng--Espinosa, with a substantially sharper treatment of the local arithmetic. In the previous Deng--Espinosa manuscript, the local analysis, incorporating the preparatory study of overorders and periodicity, is distributed over more than one hundred pages and follows the framework of Altuğ \cite{AliI}. Their final enumeration requires separate treatments of generic, dyadic, and ternary residue characteristics, together with the unit and ramified constant-term cases. By contrast, Section~\ref{sec:kloosterman-application} avoids all such classifications: averaging over the HNF parameters reduces the argument to a single linear congruence whose image can be determined explicitly.

The quadratic comparison follows from the description of quadratic
overorders and the explicit local formulas in
\cite[Proposition~3.10 and Theorem~5.4]{Malors}.  In degree three, however, overorders need not
form a chain and may be non-Gorenstein.  Yun's ideal classes can still be
partitioned by their multiplicator rings, but this partition does not by
itself determine the contribution attached to a fixed overorder.  Yun's
functional equation \cite[Theorem~2.5]{ZYun}, Lee's corresponding local
functional equations, and the equality at \(s=1\) from old version of Deng-Espinosa provide strong consistency checks,
but none of them implies Conjecture~A.

\subsection{Sketch of the proof}
The proof is local.  Proposition~\ref{prop:gorenstein-cubic-monogenic}
shows that a Gorenstein cubic \(\Z_q\)-order is monogenic unless
\(q=2\) and the order is \(\Z_2^3\).  This exceptional order is the
split maximal order.  Its occurrence is noteworthy: it is the unique
obstruction to monogenicity among local Gorenstein cubic orders, but it
is already an immediate case of the comparison and requires no separate
recursion.

For every other Gorenstein cubic order \(A\), we choose a monogenic
presentation and induct on
\[
\delta(A):=\operatorname{length}_{\Z_q}(\CO_E/A).
\]
The base \(\delta(A)=0\) is the maximal order \(A=\CO_E\).  Away from
this base, the defining cubic is not squarefree modulo \(q\).  If its
reduction has a simple root, then \(A\) is the product of \(\Z_q\) and a
quadratic order, so that branch terminates by the quadratic comparison.
Otherwise, after translating the generator, the residual polynomial is
\(X^3\).  Corollary
\ref{cor:first-overorder-terminal-cases} treats its direct terminal
alternatives.  In the sole recursive alternative, the distinguished
index-\(q\) overorder \(S\) satisfies
\[
S/qS\simeq\mathbb F_q\oplus M,
\qquad M^2=0,
\qquad \dim_{\mathbb F_q}M=2.
\]
Corollary~\ref{cor:square-zero-finite-induction} gives the first
structural step of the recurrence: it enumerates the index-\(q\)
overorders of such an \(S\) and separates them according to whether they
split, return to the monogenic Gorenstein class, or again have square-zero
special fiber.

Choose a good basis of \(S\) and write its Delone--Faddeev binary cubic form
as \(qB(X,Y)\) (see Section~\ref{subsec: parametrize-cubic}).  The projective roots of \(\bar B\) parametrize the
index-\(q\) overorders of \(S\), and the full overorder sum decomposes into
the corresponding root-indexed parts (Theorem \ref{thm:q-scaled-overorder-decomposition}).  The same construction applies to
each recursive part, while the discriminant valuation drops by \(2\).
This gives a finite recursion for the overorder zeta function.

On Yun's side, Hermite normal form converts the count of stable sublattices
into congruences governed by the same cubic polynomial (see Proposition~\ref{prop:peeled-hnf-series}  and Proposition~\ref{prop:hnf-root-count-parametrization}).  The recursion for
Yun's local zeta factor (Theorem \ref{thm:yun-local-zeta-recursion}) shows
that these congruences have the same root-indexed recursion.
It compares \(A=\Z_q[q\alpha]\) with the larger monogenic order
\(U=\Z_q[\alpha]\).  Since \([U:A]=q^3\), one has
\(\delta(U)=\delta(A)-3\), and induction on \(\delta(A)\) completes the
comparison.  The containment of the distinguished first overorder in every
strict overorder of \(A\) (Proposition \ref{prop:first-overorder-contained}) is what allows the final overorder sum to satisfy
the same induction.  This argument works in every residue characteristic.

Section  \ref{sec:kloosterman-application} concerns the evaluation of the
Dirichlet series \(D_q(z)\).   After grouping the series according to the
diagonal degree \(k=v+2r\), the Yun--overorder identity expresses the
\(k\)-th coefficient as an average of
\([t^k](1-t)J_{a,b}(t)\), where \(J_{a,b}(t)\) is Yun's local zeta factor
attached to the cubic order \(R(a,b)\).  Formula
\eqref{eq:kloosterman-hnf-formula} then identifies these coefficients with
generating functions of the sets
\(\mathfrak C_{\Phi_{a,b}}(\alpha,\beta)\).

The defining congruences of these sets reduce, after fixing the
variables \(X,Y\), to a single linear congruence.  If
\(\rho,\sigma\) denote the truncated \(q\)-adic valuations of \(X,Y\), the
image of the corresponding linear map is
\(q^{\rho+\sigma}\Z_q/q^n\Z_q\).  Hence the congruence is solvable exactly
when \(\rho+\sigma\leq v_q(c)\).  Summing the resulting fiber sizes gives
\eqref{eq:kloosterman-local-factor-rational}.  This argument is uniform in
the residue characteristic and avoids the factorization-type and
small-characteristic case distinctions required in previous approaches.

Section~\ref{sec:statement} fixes the global and local normalizations.
Section~\ref{sec:yun-exact-decomposition} recalls the binary-cubic
parametrization and records the decomposition by multiplicator ring.
Sections~\ref{sec:remaining-single-residue} and
\ref{sec:containment-hnf} reduce the proof to the square-zero first-overorder
case and derive the congruence-counting formula for Yun's local zeta
factor.  Sections~\ref{sec:primitive-induction-square-zero}-- \ref{sec:yun-local-zeta-recursion} establish the matching root-count and
overorder recursions.  The induction is completed in
Section~\ref{sec:proof-local-comparison}.  Finally,
Section~\ref{sec:kloosterman-application} averages the resulting
congruence counts and proves the uniform Euler factors of the local series $D_q(z)$.

\subsection{Further directions}
The cubic identity does not extend verbatim to arbitrary degree.  Simple
quartic examples indicate that the present overorder correction factor does
not retain enough information for a direct \(\GL_4\) analogue.  A
higher-degree comparison will therefore require a refined description of
the local contributions attached to an overorder, rather than only a formal
replacement of cubic orders by orders of rank \(n\).  Determining the
additional invariants and the recursion they satisfy is a natural next step
in understanding Yun's zeta functions for general \(\GL_n\).

This question is also motivated by Beyond Endoscopy.  Arthur's proposal
requires a higher-rank expression for finite orbital integrals that both
admits a functional equation and has coefficients suitable for Poisson
summation \cite[\S3]{Arthur}.  An appropriate generalization of the present
result could then provide the arithmetic input needed to extend the
Deng--Espinosa isolation of the trivial representation, and potentially of
other residual contributions, beyond \(\GL_3\).

\subsection*{Acknowledgements}
\addtocontents{toc}{\protect\setcounter{tocdepth}{1}}
We thank Zhaolin Li for pointing out the Proposition \ref{prop:gorenstein-cubic-monogenic}, which removes
the assumption on monogenicity of orders considered in the paper.
T. Deng is supported by  National Natural Science Foundation of
China, No. 12401013 and Beijing Natural Science Foundation, No. 1244042.
M. Espinosa thanks Jim Arthur, Bill Casselman and Clifton Cunningham for their continuous support.

\section{Conjecture A and local normalizations}\label{sec:statement}

Let \(s\) be a complex number and let \(R\) be a quadratic or cubic
\(\Z\)-order.  Write \(R^\vee:=\Hom(R,\Z)\) for the \(\Z\)-linear dual
of \(R\), viewed as an \(R\)-module.  Yun's zeta function \cite{ZYun} is
\[
J_R(s)=
\sum_{\substack{M\subseteq R^\vee\text{ an }R\text{-submodule}\\
[R^\vee:M]<\infty}}
[R^\vee:M]^{-s}.
\]
We recall the definition of overorder zeta function \(L(s,R_0)\) of
Deng--Espinosa \cite[Definition 12]{DengEspinosaGL3PS}.   Give a
prime \(q\) and an order $S$, let \(S_q\) be its completion at $q$. If \(\{\mathfrak m_i\}\) are the set of maximal ideals of  the semilocal order \(S_q\)
and \(k_i=S_q/\mathfrak m_i\), define 
\begin{equation}\label{eqn: residue-zeta}
\zeta_{S_q}(s)=\prod_i(1-|k_i|^{-s})^{-1},
\qquad
L_{S_q}(s)=\frac{\zeta_{S_q}(s)}{\zeta_{\Q_q}(s)},
\qquad
L_S(s)=\prod_q L_{S_q}(s).
\end{equation}
We also need the correction factor $h_S(s)$ which is 
multiplicative in $q$. Locally,
\[
h_{S_q}(s)=
\begin{cases}
1, & \text{if } S_q \text{ is Gorenstein},\\
1+q^{1-s}, & \text{otherwise}.
\end{cases}
\]
The dichotomy is governed by Cohen--Macaulay type,
using \cite[Proposition~3.4 and Theorem~3]{MS24}.

 Now if \(R_0\) is
a \(\Z\)-order contained in a maximal order \(\mathcal O\), define
\[
L(s,R_0)=\sum_{R_0\subseteq S\subseteq \mathcal O}
h_S(s)L_S(s)[S:R_0]^{1-2s},
\]
where the sum is over all overorders \(S\) of \(R_0\) in \(\mathcal O\).

The following is Conjecture~A from previous version of Deng--Espinosa.
\begin{aconjecture}\label{conj:identification-to-Yun}
Let \(R\) be a Gorenstein quadratic or cubic order over \(\Z\).  Then
\[
J_{R}(s)=L(s, R)\zeta_\Q(s).
\]
\end{aconjecture}

\begin{remark}
The identity in  appeared as
Conjecture~1 in an earlier version of \cite{DengEspinosaGL3PS}.  Geometrically, the identity
is related to the "H=L" conjecture of I.~Cherednik~\cite[Conjecture 4.5]{C18}, see also \cite{CH25}.
\end{remark}

\subsection{Local normalization}
\label{sec:local-normalization}

For a finite \(\Z_q\)-module \(M\), let
\(\ell_{\Z_q}(M)\) denote its length as a \(\Z_q\)-module, i.e.,
\(\#M=q^{\ell_{\Z_q}(M)}\).
Let \(R\) be an order in a number field \(E\), and let \(r_1\) and
\(r_2\) be the numbers of real and complex places of \(E\), respectively.
Using the trace pairing to identify \(E\) with its \(\Q\)-linear dual, we
identify \(R^\vee\) with the trace-dual lattice
\[
R^\vee=\{x\in E:\operatorname{Tr}_{E/\Q}(xR)\subseteq\Z\}.
\]
Following Yun \cite[\S3.2]{ZYun}, define
\[
D_R:=\#(R^\vee/R)=|\operatorname{disc}(R)|,
\qquad
\Delta_E:=\operatorname{disc}(E/\Q)=\operatorname{disc}(\CO_E).
\]
Thus \(D_R\) is the absolute discriminant of \(R\), whereas
\(\Delta_E\) is the signed discriminant of \(E\), and
\begin{equation}\label{eq:global-discriminant-index}
D_R=[\CO_E:R]^2|\Delta_E|.
\end{equation}

For a prime \(q\), let \(R_q=R\otimes_\Z\Z_q\), and
\[
\delta(R_q):=\ell_{\Z_q}(\CO_{E_q}/R_q).
\]
If \(E_q=\prod_i E_{q,i}\) and the residue degree of \(E_{q,i}/\Q_q\) is
\(n_i\), Yun's local normalization \cite[(2.5)]{ZYun} is
\[
\widetilde J_{R_q}(s)
:=
q^{\delta(R_q)s}\prod_i(1-q^{-n_i s})\,J_{R_q}(s).
\]
Let
\[
\Gamma_{E,\infty}(s):=
(\pi^{-s/2}\Gamma(s/2))^{r_1}
((2\pi)^{1-s}\Gamma(s))^{r_2}.
\]
Yun's completed zeta function and the completed Dedekind zeta function of
\(E\) are, respectively,
\[
\Lambda_R^{\mathrm{Yun}}(s)
=D_R^{s/2}\Gamma_{E,\infty}(s)J_R(s),
\qquad
\Lambda_E(s)
=|\Delta_E|^{s/2}\Gamma_{E,\infty}(s)\zeta_E(s).
\]
Yun proves \cite[(3.2)]{ZYun} that
\[
\frac{\Lambda_R^{\mathrm{Yun}}(s)}{\Lambda_E(s)}
=\prod_q\widetilde J_{R_q}(s).
\]
On the other hand, the normalization used in the overorder expression is
\[
\widetilde L_q(s,R)
:=
q^{\delta(R_q)s}
\sum_{R_q\subseteq S\subseteq\CO_{E_q}}
h_S(s)\frac{\zeta_S(s)}{\zeta_{E_q}(s)}[S:R_q]^{1-2s}.
\]
The local-global index formula used in the overorder normalization
\cite[\S 4.3.2]{DengEspinosaGL3PS} gives
\[
[\CO_E:R]^sL(s,R)\frac{\zeta_\Q(s)}{\zeta_E(s)}
=\prod_q\widetilde L_q(s,R).
\]
By \eqref{eq:global-discriminant-index}, Conjecture~A follows once one
proves the local identities
\begin{equation}\label{eq:local-yun-deng-espinosa-identity}
\widetilde J_{R_q}(s)=\widetilde L_q(s,R)
\end{equation}
for every prime \(q\).  

For the local calculations below, we suppress the subscript \(q\): thus
\(E\) is an \'etale \(\Q_q\)-algebra and
\(R\subseteq\CO_E\) is a \(\Z_q\)-order.  Put \(t=q^{-s}\).  For a
formal Laurent series \(F(t)=\sum_m f_mt^m\), write
\[
[t^n]F(t):=f_n.
\]
Yun's local zeta factor is
\begin{equation}\label{eq:yun-counting-series}
J_R(s)=\sum_{n\ge0}b_R(n)t^n,
\end{equation}
where
\[
b_R(n):=\#\{I\subseteq R^\vee:I\text{ is an }R\text{-submodule and }
[R^\vee:I]=q^n\}.
\]
Set
\[
C_R(t):=h_R(s)\zeta_R(s).
\]
For an overorder \(T\supseteq R\), let
\[
\delta(T/R):=\ell_{\Z_q}(T/R)=\delta(R)-\delta(T),
\qquad [T:R]=q^{\delta(T/R)}.
\]
Then the local overorder series is
\begin{equation}\label{eq:local-overorder-t-series}
\widetilde L_R(t):=
\sum_{R\subseteq T\subseteq\CO_E}
h_T(s)\zeta_T(s)[T:R]^{1-2s}
=\sum_{R\subseteq T\subseteq\CO_E}
q^{\delta(T/R)}t^{2\delta(T/R)}C_T(t).
\end{equation}
The two local notations are related by
\begin{equation}\label{eq:local-overorder-normalization}
\widetilde L_q(s,R)
=t^{-\delta(R)}\zeta_E(s)^{-1}\widetilde L_R(t).
\end{equation}

\section{Cubic orders and multiplicator rings}
\label{sec:yun-exact-decomposition}

\subsection{Binary cubic parametrization}
\label{subsec: parametrize-cubic}
We first recall the Delone--Faddeev parametrization.  Given
\(g\in\GL(2,\Q)\) and a binary cubic form
\[
f(X,Y)=aX^3+bX^2Y+cXY^2+dY^3,
\]
we use the action
\[
f^g(X,Y)=\det(g)f((X,Y)g^{-1}).
\]
Its discriminant
\[
\Delta(f)=b^2c^2-4ac^3-4b^3d-27a^2d^2+18abcd
\]
satisfies \(\Delta(f^g)=\det(g)^{-2}\Delta(f)\).
For a cubic
\(\Phi(T)=a_3T^3+a_2T^2+a_1T+a_0\), we write
\(\Delta(\Phi)\) for the discriminant of
\(a_3X^3+a_2X^2Y+a_1XY^2+a_0Y^3\); this is the usual polynomial
discriminant of \(\Phi\).

\begin{proposition}[Delone--Faddeev; {\cite[Proposition~4.2]{GGS}}]
\label{prop:GGS-correspondence}
There is a bijection between the \(\GL(2,\Z)\)-orbits of integral binary
cubic forms and the isomorphism classes of cubic rings.
\end{proposition}

Concretely, let \(R=\Z+\Z\alpha+\Z\beta\) and choose a good basis
\(\{1,\alpha,\beta\}\), meaning that \(\alpha\beta\in\Z\).  Such a basis
always exists \cite{DeloneFaddeev}.  The multiplication table has the
form
\begin{equation}\label{eq:delone-faddeev-multiplication-table}
\alpha^2=-ac+b\alpha-a\beta,\qquad
\alpha\beta=-ad,\qquad
\beta^2=-bd+d\alpha-c\beta,
\end{equation}
and the associated binary cubic form is
\[
f(X,Y)=aX^3+bX^2Y+cXY^2+dY^3.
\]

\begin{remark}\label{rem:functoriality-parametrization}
As observed by Deligne, the parametrization is functorial in the base
scheme \cite[p.~10]{GGS}.  In particular, replacing \(\Z\) by
\(\Z_q\) gives the corresponding parametrization of cubic
\(\Z_q\)-orders.
\end{remark}

\begin{definition}\label{def:primitive-form}
An integral binary cubic form
\(f(X,Y)=aX^3+bX^2Y+cXY^2+dY^3\) is \emph{primitive} if
\(\gcd(a,b,c,d)=1\).  A binary cubic over \(\Z_q\) is primitive if at
least one of its coefficients is a unit.
\end{definition}

\begin{proposition}[{\cite[Proposition~5.2]{GGS}}]
\label{prop:criterion-gorenstein}
An integral binary cubic form is primitive if and only if its associated
cubic ring is Gorenstein.  The same equivalence holds over \(\Z_q\).
\end{proposition}

\subsection{Gorenstein and monogenic cubic orders}
\label{subsec:gorenstein-monogenic}

The binary-cubic parametrization also makes the distinction between
Gorenstein and monogenic cubic orders explicit.

\begin{proposition}
\label{prop:gorenstein-cubic-monogenic}
Let \(R\) be a cubic \(\Z_q\)-order.
\begin{enumerate}[label=\textup{(\roman*)}, leftmargin=2em]
\item If \(q>2\) and \(R\) is Gorenstein, then \(R\) is monogenic.
\item The only Gorenstein cubic \(\Z_2\)-order that is not monogenic is
\(\Z_2^3\).
\end{enumerate}
\end{proposition}

\begin{proof}
Choose a good basis \(1,\alpha,\beta\) of \(R\), and let
\[
f(X,Y)=aX^3+bX^2Y+cXY^2+dY^3
\]
be its associated binary cubic.  For \(u,r,s\in\Z_q\), put
\(\gamma=u+r\alpha+s\beta\).  Elementary column operations, followed by
the multiplication table
\eqref{eq:delone-faddeev-multiplication-table}, give (compare
\cite[Lemma~5.3]{GGS})
\[
\begin{aligned}
\det_{\{1,\alpha,\beta\}}(1,\gamma,\gamma^2)
&=
\det
\begin{pmatrix}
1&0&*\\
0&r&br^2+ds^2\\
0&s&-ar^2-cs^2
\end{pmatrix}\\
&=-f(r,s).
\end{aligned}
\]
Consequently,
\begin{equation}\label{eq:monogenic-unit-index-form}
R=\Z_q[\gamma]
\quad\Longleftrightarrow\quad
f(r,s)\in\Z_q^\times .
\end{equation}

Suppose that \(R\) is Gorenstein.  By
Proposition~\ref{prop:criterion-gorenstein}, \(f\) is primitive, so its
reduction \(\bar f\) is a nonzero homogeneous cubic over \(\mathbb F_q\).
If \(q>2\), then
\(\#\mathbb P^1(\mathbb F_q)=q+1>3\), whereas a nonzero binary cubic has
at most three projective zeros.  Thus
\(\bar f(\bar r,\bar s)\ne0\) for some
\([\bar r:\bar s]\in\mathbb P^1(\mathbb F_q)\).  Lifting
\(\bar r,\bar s\) to \(\Z_q\) and using
\eqref{eq:monogenic-unit-index-form} proves \textup{(i)}.

Now let \(q=2\).  If \(\bar f\) does not vanish on all of
\(\mathbb P^1(\mathbb F_2)\), the same argument proves that \(R\) is
monogenic.  Otherwise, evaluating \(\bar f\) at
\([1:0]\), \([0:1]\), and \([1:1]\) gives
\[
\bar a=\bar d=0,\qquad \bar b=\bar c.
\]
Since \(\bar f\ne0\), it follows that
\[
\bar f(X,Y)=X^2Y+XY^2=XY(X+Y).
\]
Reducing \eqref{eq:delone-faddeev-multiplication-table} modulo \(2\)
then gives
\[
\bar\alpha^2=\bar\alpha,\qquad
\bar\alpha\bar\beta=0,\qquad
\bar\beta^2=\bar\beta.
\]
The three orthogonal idempotents
\(\bar\alpha,\bar\beta,1-\bar\alpha-\bar\beta\) therefore identify
\[
R/2R\simeq\mathbb F_2^3.
\]
By Hensel's lemma, these idempotents lift uniquely to \(R\), since the
derivative of \(T^2-T\) is a unit at every residual idempotent.  Their
lifts are again orthogonal and sum to \(1\), and hence give a product
decomposition of \(R\).  Each factor is a direct summand whose reduction
has dimension one over \(\mathbb F_2\), so it is free of rank one over
\(\Z_2\).  Therefore
\[
R\simeq\Z_2^3.
\]

Conversely, \(\Z_2^3\) is Gorenstein, being a product of discrete
valuation rings, but it is not monogenic.  Indeed, every element
\(\bar\gamma\in\mathbb F_2^3\) satisfies
\(\bar\gamma^2=\bar\gamma\), so the subalgebra
\(\mathbb F_2[\bar\gamma]\) has dimension at most \(2\).  Thus
\(\mathbb F_2^3\), and consequently \(\Z_2^3\), cannot be generated by
one element.
\end{proof}

\subsection{Decomposition by multiplicator ring}

We record the exact local decomposition that follows from Yun's
ideal-class formula \cite[Lemma~2.10]{ZYun}.  Let \(A\) be a \(\Z_q\)-order in an \'etale
\(\Q_q\)-algebra \(E\).
Following Yun \cite[\S2.2]{ZYun}, choose an \(\CO_E\)-generator
\(\mathfrak c\) of the different \(\mathfrak d_{E/\Q_q}\).  We realize all
local duals inside \(E\) using the modified trace pairing
\[
\langle x,y\rangle_E
:=\operatorname{Tr}_{E/\Q_q}(\mathfrak c^{-1}xy),
\qquad
T^\vee=\{x\in E:\langle x,T\rangle_E\subseteq\Z_q\}.
\]
In particular, \(\CO_E^\vee=\CO_E\).

\begin{lemma}\label{lem:multiplicator-endomorphism}
For a fractional \(A\)-ideal \(M\subset E\), let
\[
(M:M):=\{x\in E:xM\subseteq M\}.
\]
Multiplication induces a canonical isomorphism of rings
\[
(M:M)\xrightarrow{\ \sim\ }\End_A(M),
\qquad x\longmapsto(m\longmapsto xm).
\]
Moreover, \((M:M)\) is an overorder of \(A\) in \(\CO_E\), called the
\emph{multiplicator ring} of \(M\).
\end{lemma}

\begin{proof}
The homomorphism is injective because \(M\) spans \(E\) over
\(\Q_q\).  Conversely, an \(A\)-linear endomorphism of \(M\) extends over
\(\Q_q\) to an \(E\)-linear endomorphism of
\(M\otimes_{\Z_q}\Q_q\simeq E\), and hence is multiplication by a unique
element \(x\in E\).  Since it preserves \(M\), one has \(xM\subseteq M\),
so \(x\in(M:M)\).

Clearly \(A\subseteq(M:M)\).  Every \(x\in(M:M)\) acts on the finite free
\(\Z_q\)-module \(M\), so its characteristic polynomial shows that \(x\)
is integral over \(\Z_q\). Hence \(x\in\CO_E\).  Finally, \((M:M)\) embeds
in the finite \(\Z_q\)-module \(\End_{\Z_q}(M)\) and contains \(A\), so it
is an order between \(A\) and \(\CO_E\).
\end{proof}

For an overorder \(S\supseteq A\), put
\[
\Cl^{\circ}(S)
:=
E^\times\backslash
\{\,M\subset E:\ M \text{ is an } S\text{-fractional ideal and }
(M:M)=S\,\}.
\]
For \([M]\in\Cl^{\circ}(S)\), let \(J_S(s,[M])\) be Yun's summand
attached to the class \([M]\)(\cite[\S 2.9]{ZYun}).

\begin{proposition}\label{prop:yun-exact-decomposition}
For every local order \(A\),
\[
J_A(s)
=
\sum_{A\subseteq S\subseteq \CO_E}
[S:A]^{-s}
\sum_{[M]\in\Cl^{\circ}(S)} J_S(s,[M]).
\]
\end{proposition}

\begin{proof}
Every fractional \(A\)-ideal \(M\) has a multiplicator ring \((M:M)\),
which is an overorder \(S\) with \(A\subseteq S\subseteq\CO_E\).
Hence Yun's ideal classes for \(A\) decompose disjointly according to this
 multiplicator ring.

Fix such an \(S\) and a class \([M]\in\Cl^{\circ}(S)\), and scale
the representative so that \(M\subseteq S^\vee\).  The
automorphism group in Yun's summand is intrinsic to \(M\):
\[
\Aut_A(M)=\Aut_S(M)=S^\times,
\]
and the integral over \(E^\times\) involving \(M^\vee\) is the same for
the ambient orders \(A\) and \(S\).  The only change is the index factor.
Since duality reverses inclusions and
\[
[A^\vee:S^\vee]=[S:A],
\]
we have
\[
[A^\vee:M]=[A^\vee:S^\vee][S^\vee:M]=[S:A][S^\vee:M].
\]
Therefore
\[
J_A(s,[M])=[S:A]^{-s}J_S(s,[M]).
\]
Summing over all overorders $S$  and all classes in $\Cl^{\circ}(S)$ gives the formula.
\end{proof}

\section{The triple-root case}\label{sec:remaining-single-residue}

The purpose of this section is to isolate the only local case that
requires recursion.  The comparison is immediate for a maximal order and
follows from the quadratic cases when the defining cubic has a simple
root modulo \(q\).  We then analyze the remaining triple-root case, determine
its unique index-\(q\) overorder, and identify the square-zero special fiber
that underlies the arguments in the following sections.

For \(F\in \Z_q[X]\), we  denote by \(\overline F\) its reduction in \(\mathbb F_q[X]\).
\begin{proposition}
\label{prop:local-base-cases}
Let \(A=\Z_q[\alpha]\simeq\Z_q[T]/(F)\) be a monogenic cubic order in an
\'etale cubic \(\Q_q\)-algebra \(E\).  The local identity
\[
J_A(s)=\widetilde L_A(t)
\]
holds if \(A=\CO_E\), or if \(\overline F\) has a simple
\(\mathbb F_q\)-root.
\end{proposition}

\begin{proof}
If \(A=\CO_E\), then \(\CO_E^\vee=\CO_E\), and the
\(\CO_E\)-stable sublattices counted by \(J_{\CO_E}\) are precisely its
integral ideals.  Hence
\[
J_{\CO_E}(s)=\zeta_E(s).
\]
There is no strict overorder and \(h_{\CO_E}=1\), so
\(\widetilde L_{\CO_E}(t)=\zeta_E(s)\) as well.

Now suppose that \(\overline F\) has a simple root.  Hensel's lemma and
the Chinese remainder theorem give
\[
A\simeq\Z_q\times B,
\]
where \(B\) is a quadratic \(\Z_q\)-order in an \'etale quadratic
\(\Q_q\)-algebra \(K\).  The two idempotents of \(A\) split every
\(A\)-stable lattice in \(A^\vee=\Z_q\times B^\vee\), and therefore
\[
J_A(s)=\zeta_{\Z_q}(s)J_B(s).
\]
They also split every overorder of \(A\), so the overorders are exactly
\(\Z_q\times T\), with \(B\subseteq T\subseteq\CO_K\), and
\(\delta((\Z_q\times T)/A)=\delta(T/B)\).  Every quadratic order is
Gorenstein, hence
\begin{equation}\label{eq:quadratic-product-overorder-sum}
\widetilde L_A(t)=
\zeta_{\Z_q}(s)
\sum_{B\subseteq T\subseteq\CO_K}
q^{\delta(T/B)}t^{2\delta(T/B)}\zeta_T(s).
\end{equation}
We spell out the last identification, since two different zeta functions
are involved.  In \eqref{eq:quadratic-product-overorder-sum},
\(\zeta_T(s)\) is the
residue-field Euler factor defined in \eqref{eqn: residue-zeta}, which is not
the ideal-counting zeta function of \(T\).

Choose \(\Delta\in\CO_K\) so that \(1,\Delta\) is a
\(\Z_q\)-basis of \(\CO_K\).  Since \(B/\Z_q\) is a rank-one
submodule of \(\CO_K/\Z_q\), there is a unique \(n\geq 0\) such that
\[
B=\CO_n:=\Z_q[q^n\Delta]
\]
in the notation of \cite[Definition~3.3]{Malors}.  Its overorders are
precisely \(\CO_0,\ldots,\CO_n\)
\cite[Proposition~3.10 and Corollary~3.11]{Malors}.  Since
\(B\) is Gorenstein, scaling the invertible \(B\)-module \(B^\vee\) by
a generator identifies \(J_B(s)\) with the ideal-counting series
\[
J_B(s)=\sum_{I\subseteq\CO_n}[\CO_n:I]^{-s}.
\]
Here the sum is over the finite-index ideals of \(\CO_n\). This is the
series denoted by \(\zeta_n(s)\) in
\cite[Definition~2.3]{Malors}.

The residue-field factors of the orders \(\CO_i\) give directly
\begin{equation}\label{eq:quadratic-residue-factor-sum}
\sum_{i=0}^n q^{n-i}t^{2(n-i)}
\zeta_{\CO_i}(s)
=
\begin{cases}
\dfrac{\sum_{j=0}^n q^jt^{2j}}{1-t},
&K/\Q_q\text{ ramified},\\[9pt]
\dfrac{\sum_{j=0}^{n-1}q^jt^{2j}}{1-t}
 +\dfrac{q^nt^{2n}}{1-t^2},
&K/\Q_q\text{ unramified},\\[9pt]
\dfrac{\sum_{j=0}^{n-1}q^jt^{2j}}{1-t}
 +\dfrac{q^nt^{2n}}{(1-t)^2},
&K=\Q_q\times\Q_q.
\end{cases}
\end{equation}
Here a sum with upper limit \(-1\) is zero.
In the ramified case every \(\CO_i\) has residue field
\(\mathbb F_q\).  In the other two cases, \(\CO_i\) has residue field
\(\mathbb F_q\) for \(i\geq 1\), whereas \(\CO_0\) contributes
\((1-t^2)^{-1}\) in the unramified case and \((1-t)^{-2}\) in the split
case.  The three expressions in \eqref{eq:quadratic-residue-factor-sum}
are, respectively, \(R_n(t)/(1-t)\), \(U_n(t)/(1-t^2)\), and
\(S_n(t)/(1-t)^2\) in the notation of
\cite[Theorem~5.4]{Malors}.  That theorem identifies each expression
with its ideal-counting series \(\zeta_n(s)=J_B(s)\).  Therefore
\begin{equation}\label{eq:quadratic-overorder-equals-ideal-zeta}
\sum_{B\subseteq T\subseteq\CO_K}
q^{\delta(T/B)}t^{2\delta(T/B)}\zeta_T(s)=J_B(s).
\end{equation}
The term with \(T=B\) is only \(\zeta_B(s)\) in the sense of
\eqref{eqn: residue-zeta}, not \(J_B(s)\), so it does not already contain
the ideals represented by the remaining terms.  The principal-ideal
series \(\zeta_n^P(s)\) used in
\cite[Theorem~4.2]{Malors} is a third object and does not occur in
\eqref{eq:quadratic-product-overorder-sum}.  Substituting
\eqref{eq:quadratic-overorder-equals-ideal-zeta} into
\eqref{eq:quadratic-product-overorder-sum} gives
\(J_A(s)=\widetilde L_A(t)\).
\end{proof}

\begin{remark}\label{rmk:reduction-to-triple-root}
 If
\(\overline F\) is squarefree, the monogenic order
\(A\simeq\Z_q[T]/(F)\) is maximal, hence
\(A=\CO_E\).   If \(\overline F\) is not squarefree but has a simple root,
Proposition~\ref{prop:local-base-cases} applies.  Thus the only remaining
nonmaximal case is
\[
\overline F=(T-\bar a)^3.
\]
After translating the generator, we may take \(\bar a=0\).
\end{remark}

The following local form of \cite[Lemma~7]{DengEspinosaGL3PS} will be
used repeatedly.
\begin{lemma}\label{lem:overorders-by-truncated-ideals}
Let \(A\) be a \(\Z_q\)-order in an \'etale \(\Q_q\)-algebra \(E\), and
let \(r\ge0\).  For an overorder \(S\supseteq A\) with
\([S:A]=q^r\), set \(I_S:=q^rS\subseteq A\).  Then
\[
q^rA\subseteq I_S\subseteq A,\qquad
\ell_{\Z_q}(I_S/q^rA)=r,\qquad
I_S^2\subseteq q^rI_S.
\]
Conversely, every \(A\)-submodule \(I\) satisfying these conditions
arises uniquely from the overorder
\[
S=q^{-r}I=\{x\in q^{-r}A:q^rx\in I\}.
\]
\end{lemma}

We apply the lemma to determine the index-\(q\) overorders in the remaining
triple-root case.

\begin{proposition}\label{prop:first-overorder-single-residue}
Let
\[
A=\Z_q[\beta]\simeq \Z_q[X]/(F(X))
\]
be a monogenic local cubic order such that
\[
F(X)=X^3+quX^2+qvX+qw,\qquad u,v,w\in\Z_q .
\]
If \(w\in\Z_q^\times\), then \(A\) is maximal.  If \(w=qw_1\), then \(A\)
has a unique index-\(q\) overorder
\[
S_1=\Z_q\langle 1,\beta,\delta\rangle,\qquad
\delta:=\frac{\beta^2+qu\beta+qv}{q}.
\]
In \(S_1\) one has
\begin{equation}\label{eq:first-overorder-relations}
\beta^2=q\delta-qu\beta-qv,\qquad
\beta\delta=-qw_1,\qquad
\delta^2=v\delta-w_1\beta-quw_1 .
\end{equation}
Let \(\bar v,\bar w_1\in\mathbb F_q\) denote the reductions of \(v,w_1\)
modulo \(q\), respectively.  We use the same notation for their images in
\(S_1/qS_1\).  There are three mutually exclusive cases:

\begin{enumerate}[label=\textup{(\alph*)}, leftmargin=2em]
\item If \(\bar v\ne0\), then
\[
S_1\simeq \Z_q\times\CO_K
\]
for a ramified quadratic \(\Q_q\)-algebra \(K\).  In particular, \(S_1\)
is maximal.
\item If \(\bar v=0\) and \(\bar w_1\ne0\), then \(S_1\) is monogenic and
maximal.  More precisely, \(S_1=\Z_q[\delta]\),
\[
S_1/qS_1\simeq \mathbb F_q[\bar\delta]/(\bar\delta^3).
\]
The element \(\delta\) satisfies
\begin{equation}\label{eq:first-overorder-eisenstein-polynomial}
\delta^3-v\delta^2+quw_1\delta-qw_1^2=0,
\end{equation}
and the polynomial in \eqref{eq:first-overorder-eisenstein-polynomial} is
Eisenstein.
\item If \(\bar v=\bar w_1=0\), then \(S_1/qS_1\) is local with square-zero
maximal ideal of dimension \(2\) as $\mathbb F_q$ vector space.  In particular, \(S_1\) is not Gorenstein.
\end{enumerate}
\end{proposition}

\begin{proof}
If \(w\) is a unit, then \(F\) is Eisenstein.  Hence
\(\Z_q[X]/(F)\) is a discrete valuation ring and is the integral closure of
\(\Z_q\) in \(\Q_q[X]/(F)\)(see
\cite[Chapter~I, \S6, Proposition~17 and its corollary]{Serre}).
Thus \(A\) is maximal.

Assume now \(w=qw_1\), so \(\overline F=X^3\) in \(\mathbb F_q\).  Let
\(\bar\beta\) denote the image of \(\beta\) in
\(A/qA\simeq\mathbb F_q[\bar\beta]/(\bar\beta^3)\).  By
Lemma~\ref{lem:overorders-by-truncated-ideals} with \(r=1\), an index-\(q\)
overorder is equivalent to an \(A\)-submodule \(I\) such that
\(qA\subset I\subset A\), \(\dim_{\mathbb F_q}(I/qA)=1\), and
\(I^2\subset qI\), where the corresponding overorder is \(q^{-1}I\).  Thus
\(L:=I/qA\) is a square-zero line in \(A/qA\).  Indeed,
\(x=a+b\bar\beta+c\bar\beta^2\) satisfies \(x^2=0\) only if \(a=b=0\).
Hence \(L=\mathbb F_q\bar\beta^2\).  Its inverse image is
\(I=qA+\Z_q(\beta^2+qu\beta+qv)\).  A direct calculation using
\(F(\beta)=0\) shows that \(I^2\subset qI\) holds exactly when
\(F(0)=qw\in q^2\Z_q\), as is the case because \(w=qw_1\).  Thus the
index-\(q\) overorder \(q^{-1}I\) is unique and is obtained by adjoining
\[
\delta=\frac{\beta^2+qu\beta+qv}{q}.
\]
Using \(F(\beta)=0\), one computes
\[
\beta^2=q\delta-qu\beta-qv,\qquad
\beta\delta
=\frac{\beta^3+qu\beta^2+qv\beta}{q}
=-w=-qw_1.
\]
Multiplying the defining equation \(q\delta=\beta^2+qu\beta+qv\) by
\(\delta\), and using \(\beta\delta=-qw_1\), gives
\[
q\delta^2
=\beta^2\delta+qu\beta\delta+qv\delta
=-qw_1\beta-q^2uw_1+qv\delta,
\]
hence the third relation in \eqref{eq:first-overorder-relations}.  These
relations show that \(S_1\) is closed under multiplication.

Denote by \(\bar\beta,\bar\delta\in S_1/qS_1\) the residue classes of
\(\beta,\delta\), respectively.  Reducing
\eqref{eq:first-overorder-relations} modulo \(q\) gives, in \(S_1/qS_1\),
\[
\bar\beta^2=0,\qquad
\bar\beta\bar\delta=0,\qquad
\bar\delta^2=\bar v\,\bar\delta-\bar w_1\,\bar\beta .
\]
If \(\bar v\ne0\), put
\[
H(Y):=qY^3+quY^2+vY+w_1.
\]
Its reduction modulo $q$ is the separable linear polynomial
\(\bar vY+\bar w_1\).  Hensel's lemma gives \(y\in\Z_q\) with
\(H(y)=0\), and therefore \(a:=qy\) is a root of \(F\).  Hence
\[
F(X)=(X-a)G(X),\qquad G(X)=X^2+bX+c.
\]
Comparison of coefficients gives
\[
b=qu+a\in q\Z_q,
\qquad
c=qv+ab\in q\Z_q^\times.
\]
Thus \(G\) is Eisenstein, and if \(K=\Q_q[X]/(G)\), then
\(\Z_q[X]/(G)=\CO_K\) by
\cite[Chapter~I, \S6, Proposition~17 and its corollary]{Serre}.
The product order \(\Z_q\times\CO_K\) contains \(A\).  Indeed, for
\(I=(X-a)\) and \(J=(G)\) in \(\Z_q[X]\), one has
\(I\cap J=IJ=(F)\) and
\(\Z_q[X]/(I+J)\cong\Z_q/(G(a))\).  The standard exact sequence
\[
0\longrightarrow \Z_q[X]/(F)
 \longrightarrow \Z_q[X]/(X-a)\times\Z_q[X]/(G)
 \longrightarrow \Z_q/(G(a))\longrightarrow 0
\]
therefore gives
\[
\bigl[\Z_q[X]/(X-a)\times\Z_q[X]/(G):\Z_q[X]/(F)\bigr]
=\#\bigl(\Z_q/(G(a))\bigr)
=q^{v_q(G(a))}=q,
\]
because \(G(a)\equiv c\pmod {q^2}\) and \(v_q(c)=1\).  The uniqueness of
the index-\(q\) overorder proved above therefore identifies this product
order with \(S_1\).  Hence \(S_1\) is maximal.

If \(\bar v=0\) and \(\bar w_1\ne0\), then the third relation in
\eqref{eq:first-overorder-relations} gives
\begin{equation}
\label{eqn: relation-beta-w_1}
\beta=\frac{v\delta-\delta^2-quw_1}{w_1}.
\end{equation}
In \(S_1/qS_1\), the preceding identity gives
\[
\bar\beta=-\bar w_1^{-1}\bar\delta^2,\qquad \bar\delta^3=0.
\]
Thus \(S_1/qS_1\) is generated by \(\bar\delta\) and is isomorphic to
\(\mathbb F_q[\bar\delta]/(\bar\delta^3)\).  Nakayama's lemma gives
\(S_1=\Z_q[\delta]\). 

Substituting \eqref{eqn: relation-beta-w_1} into \(\beta\delta=-qw_1\) gives
\eqref{eq:first-overorder-eisenstein-polynomial}.  Because
\(v\in q\Z_q\) and \(w_1\in\Z_q^\times\), this polynomial is Eisenstein.
The first paragraph of the proof, applied to \(\delta\), shows that
\(\Z_q[\delta]=S_1\) is maximal.

Finally, if \(\bar v=\bar w_1=0\), then
$(\bar\beta,\bar\delta)^2=0$
and the two-dimensional maximal ideal of \(S_1/qS_1\) is its socle.  The
Artinian socle criterion \cite[Proposition~21.5]{Eis13} therefore shows that
\(S_1/qS_1\) is not Gorenstein.  Since \(S_1\) is \(\Z_q\)-torsion-free,
\(q\) is \(S_1\)-regular.  If \(S_1\) were Gorenstein, the
regular-element criterion \cite[Proposition~3.1.19]{BrunsHerzog} would imply
that \(S_1/qS_1\) is Gorenstein, a contradiction.  Hence \(S_1\) is not
Gorenstein.
\end{proof}

\begin{definition}\label{def:square-zero-special-fiber}
A local cubic \(\Z_q\)-order \(S\) has \emph{square-zero special fiber} if
\(S/qS\) is local with residue field \(\mathbb F_q\) and its maximal ideal
\(M\) satisfies
\[
M^2=0,
\qquad
\dim_{\mathbb F_q}M=2.
\]
Equivalently, as an \(\mathbb F_q\)-algebra,
\[
S/qS\simeq\mathbb F_q\oplus M
\]
where the multiplication on the right is
\[
(a,m)(b,n)=(ab,an+bm).
\]
For an order \(A\) as in
Proposition~\ref{prop:first-overorder-single-residue} that has the unique
index-\(q\) overorder \(S_1\) described there, we say that \(A\) is in the
\emph{square-zero first-overorder case} if \(S_1\) has square-zero special
fiber.
\end{definition}

\begin{corollary}\label{cor:first-overorder-terminal-cases}
In cases \textup{(a)} and \textup{(b)} of
Proposition~\ref{prop:first-overorder-single-residue}, the local 
identity \eqref{eq:local-yun-deng-espinosa-identity} holds for the
original order \(A\).
\end{corollary}

\begin{proof}
In both cases the unique index-\(q\) overorder is \(\CO_E\), so
\(\delta(A)=1\).  Let
\[
P_A^{\mathrm{Yun}}(t):=\zeta_E(s)^{-1}J_A(s),
\qquad
P_A^{\mathrm{DS}}(t):=\zeta_E(s)^{-1}\widetilde L_A(t).
\]
Yun's functional equation \cite[Theorem~2.5]{ZYun} shows that
\(P_A^{\mathrm{Yun}}\) has degree at most \(2\) and satisfies
\[
P(t)=qt^2P(q^{-1}t^{-1}).
\]

There are only two overorders, \(A\) and \(\CO_E\).  Since \(A\) is
Gorenstein with one degree-one residue factor,
\eqref{eq:local-overorder-t-series} gives
\[
\widetilde L_A(t)=\frac1{1-t}+qt^2\zeta_E(s).
\]
In case \textup{(a)}, one has
\(\zeta_E(s)=(1-t)^{-2}\), and in case \textup{(b)}, one has
\(\zeta_E(s)=(1-t)^{-1}\).  Hence
\[
P_A^{\mathrm{DS}}(t)=
\begin{cases}
1-t+qt^2,&\text{in case \textup{(a)}},\\
1+qt^2,&\text{in case \textup{(b)}}.
\end{cases}
\]
Both polynomials satisfy the palindromic relation
\(P(t)=qt^2P(q^{-1}t^{-1})\).

Finally, \(A^\vee\) is an invertible \(A\)-module because \(A\) is
monogenic.  After scaling by a generator, the colength-zero and
colength-one submodules counted by \(J_A\) are \(A\) and its maximal
ideal.  Thus
\[
[t^0]J_A(s)=[t^1]J_A(s)=1.
\]
The same two coefficients of \(\widetilde L_A(t)\) are \(1\).  Since
\(\zeta_E(s)^{-1}\) has constant term \(1\), the constant and linear
coefficients of
\(P_A^{\mathrm{Yun}}-P_A^{\mathrm{DS}}\) vanish.  Its palindromic
relation forces the quadratic coefficient to vanish as well.  Therefore
\(J_A(s)=\widetilde L_A(t)\).
\end{proof}
\begin{proposition}\label{prop:square-zero-index-q-overorders}
Keep the notation of Proposition~\ref{prop:first-overorder-single-residue},
and assume that we are in case \textup{(c)} there.  Then
\(S_1\) has square-zero special fiber in the sense of
Definition~\ref{def:square-zero-special-fiber}.  Thus
\[
v=qv_1,\qquad w_1=qw_2
\]
with \(v_1,w_2\in\Z_q\), and
\[
S_1=\Z_q\langle 1,\beta,\delta\rangle,\qquad
\delta=\frac{\beta^2+qu\beta+qv}{q}
\]
satisfies
\[
S_1/qS_1\simeq
\mathbb F_q\oplus
\mathbb F_q\bar\beta\oplus
\mathbb F_q\bar\delta,
\qquad
(\bar\beta,\bar\delta)^2=0 .
\]
Let
\[
\Phi_{S_1}(Y,Z):=
Y^3+\bar uY^2Z+\bar v_1YZ^2+\bar w_2Z^3
\in\mathbb F_q[Y,Z].
\]
Then the index-\(q\) overorders of \(S_1\) are in natural bijection with
the projective roots of \(\Phi_{S_1}\):
\[
\{\,S:\ S_1\subseteq S\subseteq\CO_E,\ [S:S_1]=q\,\}
\longleftrightarrow
\{\,[a:b]\in\mathbb P^1(\mathbb F_q):\Phi_{S_1}(a,b)=0\,\}.
\]
For a root \([a:b]\), choose lifts \(a,b\in\Z_q\), set
\[
g_{a,b}:=a\beta+b\delta,\qquad
I_{a,b}:=qS_1+\Z_q g_{a,b},
\]
and the corresponding overorder is
\[
S_{a,b}=q^{-1}I_{a,b}=S_1+\Z_q\frac{g_{a,b}}{q}.
\]
\end{proposition}

\begin{proof}
By Lemma~\ref{lem:overorders-by-truncated-ideals} with \(A=S_1\) and
\(r=1\), an index-\(q\) overorder of \(S_1\) is equivalent to an
\(S_1\)-submodule \(I\) satisfying
\[
qS_1\subset I\subset S_1,\qquad \ell_{\Z_q}(I/qS_1)=1,\qquad
I^2\subset qI .
\]
The algebra \(S_1/qS_1\) has square-zero maximal ideal
\[
\mathfrak n=(\bar\beta,\bar\delta).
\]
The one-dimensional ideals of \(S_1/qS_1\) are exactly the lines in
\(\mathfrak n\).  Therefore every possible \(I\) has the form
\[
I=qS_1+\Z_qg_{a,b},\qquad g_{a,b}=a\beta+b\delta,
\]
for a unique point \([a:b]\in\mathbb P^1(\mathbb F_q)\).

It remains to impose \(I^2\subset qI\).  In the square-zero case, the
relations \eqref{eq:first-overorder-relations} become
\[
\beta^2=q\delta-qu\beta-q^2v_1,\qquad
\beta\delta=-q^2w_2,\qquad
\delta^2=qv_1\delta-qw_2\beta-q^2uw_2 .
\]
Hence, modulo \(q^2S_1\),
\[
g_{a,b}^2
\equiv
q\left(
(-a^2u-b^2w_2)\beta+(a^2+b^2v_1)\delta
\right).
\]
The condition \(g_{a,b}^2\in qI_{a,b}\) is therefore equivalent to the
condition that the vector
\[
(-a^2\bar u-b^2\bar w_2,\ a^2+b^2\bar v_1)
\]
be proportional to \((a,b)\) in the two-dimensional vector space
\(\mathfrak n\).  Taking the determinant gives
\[
b(-a^2\bar u-b^2\bar w_2)-a(a^2+b^2\bar v_1)=0,
\]
or equivalently
\[
a^3+\bar u a^2b+\bar v_1ab^2+\bar w_2b^3=0.
\]
This is precisely \(\Phi_{S_1}(a,b)=0\).  Thus the projective roots of
\(\Phi_{S_1}\) are exactly the lines producing index-\(q\) overorders, and
the formula \(S_{a,b}=q^{-1}I_{a,b}\) is the inverse construction in
Lemma~\ref{lem:overorders-by-truncated-ideals}.
\end{proof}

We shall use the following basis-independent form of this calculation.

\begin{proposition}\label{prop:universal-square-zero-branching}
Let \(S\) be a local cubic \(\Z_q\)-order with square-zero special fiber.  We have
\[
S/qS\simeq \mathbb F_q\oplus M,\qquad M^2=0,\qquad \dim_{\mathbb F_q}M=2.
\]
Choose a \(\Z_q\)-basis \(1,x,y\) of \(S\) whose reductions
\(\bar x,\bar y\) form a basis of \(M\), and write
\begin{align*}
x^2&=q(a_0+a_1x+a_2y),\\
xy&=q(b_0+b_1x+b_2y),\\
y^2&=q(c_0+c_1x+c_2y),
\end{align*}
with \(a_i,b_i,c_i\in\Z_q\).  For \([A:B]\in\mathbb P^1(\mathbb F_q)\),
choose lifts \(A,B\in\Z_q\), put
\[
g_{A,B}:=Ax+By,\qquad I_{A,B}:=qS+\Z_qg_{A,B},
\]
and the quadratic forms
\[
\Gamma_i(A,B):=A^2a_i+2ABb_i+B^2c_i\qquad (i=0,1,2).
\]
Then \(I_{A,B}\) gives an index-\(q\) overorder of \(S\)
\[
S_{[A:B]}:=q^{-1}I_{A,B}=S+\Z_q\frac{g_{A,B}}q,
\]
if and only if
\begin{equation}\label{eq:universal-square-zero-branching}
\overline{\Gamma_0(A,B)}=0,\qquad
B\,\overline{\Gamma_1(A,B)}-A\,\overline{\Gamma_2(A,B)}=0 .
\end{equation}
Moreover, every index-\(q\) overorder of \(S\) is obtained uniquely in this
way.
\end{proposition}

\begin{proof}
As in the proof of Proposition~\ref{prop:square-zero-index-q-overorders},
Lemma~\ref{lem:overorders-by-truncated-ideals} identifies the possible
index-\(q\) overorders with the lines in \(M\).  The line generated by
\(A\bar x+B\bar y\) is admissible precisely when the class of
\[
q^{-1}g_{A,B}^2
=\Gamma_0(A,B)+\Gamma_1(A,B)x+\Gamma_2(A,B)y
\]
lies in that line.  Its scalar component must vanish, and its vector
component must be proportional to \((A,B)\).  These are exactly the two
conditions in \eqref{eq:universal-square-zero-branching}.  The uniqueness
and the formula for \(S_{[A:B]}\) follow from the same lemma.
\end{proof}

\begin{proposition}\label{prop:universal-square-zero-index-q-overorder-type}
Keep the hypotheses and notation of
Proposition~\ref{prop:universal-square-zero-branching}.  Let
\([A:B]\in\mathbb P^1(\mathbb F_q)\) satisfy
\eqref{eq:universal-square-zero-branching}, put
\[
g:=Ax+By,
\]
and choose
\[
h:=Cx+Dy
\]
so that \(g,h\) reduce to an \(\mathbb F_q\)-basis of \(M\).  Let
\begin{align*}
g^2&=q(\gamma_0+\gamma_1g+\gamma_2h),\\
gh&=q(\delta_0+\delta_1g+\delta_2h),\\
h^2&=q(\epsilon_0+\epsilon_1g+\epsilon_2h),
\end{align*}
with coefficients in \(\Z_q\).  The condition
\eqref{eq:universal-square-zero-branching} is equivalent to
\[
\gamma_0,\gamma_2\in q\Z_q .
\]
Write
\[
\gamma_0=q\gamma_0^+,\qquad \gamma_2=q\gamma_2^+,
\]
and
\[
\eta:=g/q, \qquad S':=S+\Z_q\eta .
\]
Thus \(S'\) is the index-\(q\) overorder corresponding to \([A:B]\).
Then \(1,h,\eta\) is a \(\Z_q\)-basis of \(S'\), and in \(S'/qS'\) one has
\begin{equation}\label{eq:universal-index-q-overorder-residual-table}
\bar h^2=0,\qquad
\bar h\bar\eta=\bar\delta_2\bar h,\qquad
\bar\eta^2=\bar\gamma_0^++\bar\gamma_1\bar\eta
+\bar\gamma_2^+\bar h .
\end{equation}
Moreover
\begin{equation}\label{eq:universal-index-q-overorder-root}
\bar\delta_2^2-\bar\gamma_1\bar\delta_2-\bar\gamma_0^+=0.
\end{equation}
Let
\[
P(T):=T^2-\bar\gamma_1T-\bar\gamma_0^+\in\mathbb F_q[T].
\]
Then \(\bar\delta_2\) is a root of \(P\), whose derivative is
\(P'(T)=2T-\bar\gamma_1\).  Thus:
\begin{enumerate}[label=\textup{(\alph*)}, leftmargin=2em]
\item If \(P'(\bar\delta_2)\ne0\), then \(S'\) has a nontrivial
idempotent; equivalently \(S'\simeq\Z_q\times B\) for a
quadratic \(\Z_q\)-order \(B\).
\item If \(P'(\bar\delta_2)=0\), set
\[
\bar\varepsilon:=\bar\eta-\bar\delta_2 .
\]
Then
\begin{equation}\label{eq:universal-multiple-root-index-q-overorder}
\bar h^2=0,\qquad \bar h\bar\varepsilon=0,\qquad
\bar\varepsilon^2=\bar\gamma_2^+\bar h .
\end{equation}
If \(\bar\gamma_2^+\ne0\), then
\[
S'/qS'\simeq \mathbb F_q[\bar\varepsilon]/(\bar\varepsilon^3),
\]
so \(S'\) is monogenic and Gorenstein.  If \(\bar\gamma_2^+=0\), then
\[
S'/qS'\simeq \mathbb F_q\oplus
\mathbb F_q\bar h\oplus\mathbb F_q\bar\varepsilon,
\qquad
(\bar h,\bar\varepsilon)^2=0,
\]
so \(S'\) again has square-zero special fiber.
\end{enumerate}
\end{proposition}

\begin{proof}
The condition \eqref{eq:universal-square-zero-branching} says precisely
that, modulo \(q\), the class of \(q^{-1}g^2\) has zero scalar component and
lies in the line spanned by \(g\).  In the basis \(1,g,h\), this is
equivalent to \(\gamma_0,\gamma_2\in q\Z_q\).

Since \(g=q\eta\), the three product formulae in the statement give exact
relations in \(S'\):
\[
h^2=q\epsilon_0+q^2\epsilon_1\eta+q\epsilon_2h,
\]
\[
h\eta=\delta_0+q\delta_1\eta+\delta_2h,
\]
and
\[
\eta^2=\gamma_0^++\gamma_1\eta+\gamma_2^+h.
\]
Modulo \(q\), the first relation gives \(\bar h^2=0\).  The second relation
gives
\(\bar h\bar\eta=\bar\delta_0+\bar\delta_2\bar h\); multiplying by
\(\bar h\) yields \(\bar\delta_0\bar h=0\).  The elements
\(1,\bar h,\bar\eta\) form an \(\mathbb F_q\)-basis of \(S'/qS'\), so
\(\bar h\ne0\).  Since \(\bar\delta_0\in\mathbb F_q\), any nonzero
\(\bar\delta_0\) would be a unit in \(S'/qS'\); hence
\(\bar\delta_0=0\).  This gives
\eqref{eq:universal-index-q-overorder-residual-table}.

By associativity,
\[
\bar h(\bar\eta^2)=(\bar h\bar\eta)\bar\eta .
\]
Using \eqref{eq:universal-index-q-overorder-residual-table}, the left side is
\[
(\bar\gamma_0^++\bar\gamma_1\bar\delta_2)\bar h,
\]
while the right side is \(\bar\delta_2^2\bar h\).  This proves
\eqref{eq:universal-index-q-overorder-root}, so \(\bar\delta_2\) is a root of
\(P\).

Put \(\bar\varepsilon=\bar\eta-\bar\delta_2\).  Since
\(\bar h\bar\eta=\bar\delta_2\bar h\), one has
\[
\bar h\bar\varepsilon=0.
\]
Also, using \(P(\bar\delta_2)=0\),
\[
\bar\varepsilon^2
=
(\bar\gamma_1-2\bar\delta_2)\bar\varepsilon
+\bar\gamma_2^+\bar h .
\]
If \(P'(\bar\delta_2)=2\bar\delta_2-\bar\gamma_1\ne0\), then
\(a:=\bar\gamma_1-2\bar\delta_2\) is nonzero.  The element
\[
e:=a^{-1}\bar\varepsilon+a^{-2}\bar\gamma_2^+\bar h
\]
satisfies \(e^2=e\), and it is neither \(0\) nor \(1\).  Thus
\(S'/qS'\) has a nontrivial idempotent.  Idempotents lift over the complete
local base \(\Z_q\), so \(S'\) decomposes as \(\Z_q\times B\), with \(B\)
quadratic.

If \(P'(\bar\delta_2)=0\), then the same formula reduces to
\[
\bar\varepsilon^2=\bar\gamma_2^+\bar h,
\]
which is \eqref{eq:universal-multiple-root-index-q-overorder}.  If
\(\bar\gamma_2^+\ne0\), then
\[
\bar h=(\bar\gamma_2^+)^{-1}\bar\varepsilon^2
\]
and \(\bar\varepsilon^3=0\), so the residual algebra is
\(\mathbb F_q[\bar\varepsilon]/(\bar\varepsilon^3)\).  Nakayama's lemma
then gives \(S'=\Z_q[\varepsilon]\), so \(S'\) is monogenic, hence
Gorenstein.  If \(\bar\gamma_2^+=0\), then
\eqref{eq:universal-multiple-root-index-q-overorder} says exactly that
the maximal ideal spanned by \(\bar h,\bar\varepsilon\) has square zero.
\end{proof}

The preceding two propositions give the first layer of the overorder
recursion.  We record their combined conclusion in the following form.

\begin{corollary}
\label{cor:square-zero-finite-induction}
Let \(S\) have square-zero special fiber, and retain the notation of
Proposition~\ref{prop:universal-square-zero-branching}.  Its index-\(q\)
overorders are precisely the orders \(S_{[A:B]}\) attached to the points
\([A:B]\in\mathbb P^1(\mathbb F_q)\) satisfying
\eqref{eq:universal-square-zero-branching}.  These overorders fall into
exactly three mutually exclusive classes:
\begin{enumerate}[label=\textup{(\roman*)}, leftmargin=2em]
\item \(S'\simeq\Z_q\times B\) for a quadratic \(\Z_q\)-order \(B\);
\item \(S'\) is monogenic and Gorenstein, with
\[
S'/qS'\simeq\mathbb F_q[\epsilon]/(\epsilon^3);
\]
\item \(S'\) again has square-zero special fiber.
\end{enumerate}
Thus the third class is the only one for which the same square-zero
first-layer classification must be repeated.
\end{corollary}

\begin{proof}
Proposition~\ref{prop:universal-square-zero-branching} lists exactly the
index-\(q\) overorders of \(S\), and
Proposition~\ref{prop:universal-square-zero-index-q-overorder-type} gives
the three alternatives above.  They are mutually exclusive: the first
has nonlocal special fiber, the second has local Gorenstein special fiber,
and the third has local non-Gorenstein special fiber.  Only the last
alternative remains in the class of orders with square-zero special fiber.
\end{proof}

Corollary~\ref{cor:square-zero-finite-induction} is the starting point,
rather than the whole, of the recurrence.  The product case reduces to the
quadratic comparison in Proposition~\ref{prop:local-base-cases}, and
the second case returns to the monogenic Gorenstein class treated by
induction.  The third case is the genuinely repeated square-zero case.
Moreover, as Remark~\ref{rmk:higher-index-covers} explains, a finite-index
overorder need not be reachable through a chain of index-\(q\) overorders;
Section~\ref{sec:overorder-recursion} incorporates these higher-index
overorders into the full recurrence.

\begin{remark}\label{rmk:poonen-residual-types}
Put \(k=\overline{\mathbb F}_q\).  After extension of scalars from
\(\mathbb F_q\) to \(k\), the local rank-three algebras occurring in the
preceding classification are precisely the two rank-three entries in
Poonen's table of local finite-rank algebras
\cite[Table~1]{Poonen}.  If \(\mathfrak m\) denotes the
maximal ideal and
\[
\vec d=\bigl(\dim_k\mathfrak m^i/\mathfrak m^{i+1}\bigr)_{i\geq 1},
\]
they are
\[
k[\epsilon]/(\epsilon^3),\qquad \vec d=(1,1),
\]
and
\[
k[\epsilon_1,\epsilon_2]/(\epsilon_1,\epsilon_2)^2,
\qquad \vec d=(2).
\]
These are, respectively, the curvilinear (equivalently, monogenic) type
and the square-zero type used above.  The nonlocal possibilities for
\((S/qS)\otimes_{\mathbb F_q}k\) are products of lower-rank local
algebras, namely \(k^3\) and
\(k\times k[\epsilon]/(\epsilon^2)\).
\end{remark}

\begin{remark}\label{rmk:higher-index-covers}
Corollary~\ref{cor:square-zero-finite-induction} describes only the
index-\(q\) overorders of \(S\).  It does not assert that every finite-index
overorder \(T\supset S\) admits a chain
\[
S=S_0\subset S_1\subset\cdots\subset S_r=T,
\qquad [S_i:S_{i-1}]=q\quad(1\le i\le r).
\]
If this chain property held, repeated application of
Corollary~\ref{cor:square-zero-finite-induction}, together with the product
and monogenic cases, would recover the entire overorder interval from its
index-\(q\) layers.  This is what happens for a quadratic order
\(\CO_n\) (see the proof of Proposition \ref{prop:local-base-cases}): its overorders are the linearly ordered sequence
\[
\CO_n\subset\CO_{n-1}\subset\cdots\subset\CO_0,
\qquad [\CO_{i-1}:\CO_i]=q,
\]
by \cite[Proposition~3.10 and Corollary~3.11]{Malors}.  The resulting
quadratic recursion therefore proceeds along a single chain.

The chain property fails for cubic orders.  For example, retain the notation of
Proposition~\ref{prop:square-zero-index-q-overorders} and suppose that
\(\Phi_{S_1}\) has no zero in \(\mathbb P^1(\mathbb F_q)\).  Then \(S_1\)
has no index-\(q\) overorder.  Nevertheless,
\[
U:=\Z_q\left\langle 1,\frac{\beta}{q},\frac{\delta}{q}\right\rangle
\]
is an overorder of \(S_1\):
\[
\left(\frac{\beta}{q}\right)^2
=\frac{\delta}{q}-u\frac{\beta}{q}-v_1,
\qquad
\frac{\beta}{q}\frac{\delta}{q}=-w_2,
\qquad
\left(\frac{\delta}{q}\right)^2
=v_1\frac{\delta}{q}-w_2\frac{\beta}{q}-uw_2.
\]
Here \([U:S_1]=q^2\).  Thus the first index-\(q\) layer can be empty even
though the full overorder interval is not.  Consequently, the cubic
recursion cannot be reconstructed by iterating the index-\(q\)
classification alone: it must also include such higher-index jumps. This is the essential structural difference from
the quadratic, single-chain recurrence. For further discussions, see Section \ref{sec:overorder-recursion}. 
\end{remark}

\section{Containment and Hermite normal form}
\label{sec:containment-hnf}

Let \(A\) be in the nonmaximal triple-root situation of
Proposition~\ref{prop:first-overorder-single-residue}, and let \(S_1\) be
its unique index-\(q\) overorder.  Proposition
\ref{prop:first-overorder-contained} proves that every strict overorder of
\(A\) contains \(S_1\).  Since \([S_1:A]=q\), the sum over strict
overorders in \eqref{eq:local-overorder-t-series} is then
\(qt^2\widetilde L_{S_1}(t)\).  In the square-zero first-overorder case of
Definition~\ref{def:square-zero-special-fiber}, this is the containment
used in the induction of Section~\ref{sec:proof-local-comparison}.

We then parametrize the ideals of an arbitrary monogenic cubic order by
Hermite normal form, obtain the three explicit congruences
\eqref{eq:general-hnf-1}--\eqref{eq:general-hnf-3}, and specialize them to
the square-zero case.

\begin{proposition}\label{prop:first-overorder-contained}
Let
\[
A=\Z_q[\beta]\simeq\Z_q[X]/(F),\qquad
F(X)=X^3+quX^2+qvX+q^2w_1.
\]
Let
\[
\vartheta:=\beta^2+qu\beta+qv,
\qquad S_1:=A+\Z_q\frac{\vartheta}{q}.
\]
Then every strict overorder of \(A\) contains \(S_1\).
\end{proposition}

\begin{proof}
Proposition~\ref{prop:first-overorder-single-residue} identifies \(S_1\)
as the unique index-\(q\) overorder of \(A\) and gives three cases.  In
cases \textup{(a)} and \textup{(b)} of that proposition, \(S_1\) is
maximal, hence \(S_1=\CO_E\).  Since \([S_1:A]=q\), every strict
overorder of \(A\) is then equal to \(S_1\).

It remains to consider case \textup{(c)}.  Write
\[
v=qv_1,
\qquad
w_1=qw_2.
\]
Then \(\alpha:=\beta/q\) is integral, because it satisfies the monic
polynomial
\[
\frac{F(qT)}{q^3}=T^3+uT^2+v_1T+w_2\in\Z_q[T].
\]
Moreover,
\[
\frac{\vartheta}{q}=q(\alpha^2+u\alpha+v_1)
\]
is integral.

Let \(T\supsetneq A\) be an overorder.  Choose \(y\in T\setminus A\) and
the least \(m\ge1\) such that \(q^my\in A\).  Then
\(x=q^{m-1}y\) belongs to \(T\cap q^{-1}A\) but not to \(A\).   Since \(qx\in A\), there
are unique \(c,r,s\in\Z_q\) such that
\[
qx=c+r\beta+s\vartheta.
\]
Since \(x\), \(\alpha=\beta/q\),
and \(\vartheta/q\) are integral, the preceding equality gives
\[
\frac cq=x-r\alpha-s\frac{\vartheta}{q}.
\]
Thus \(c/q\in\Q_q\) is integral over \(\Z_q\), so \(c/q\in\Z_q\) and
\(c\in q\Z_q\).  After subtracting this scalar from \(x\), we assume
\[
x=\frac{r\beta+s\vartheta}{q}.
\]

If \(r\in q\Z_q\), then \(s\) is a unit because \(x\notin A\), and it follows immediately that \(\vartheta/q\in T\).  Suppose that \(r\)
is a unit.  The relations
\[
\frac{\beta^2}{q}=\frac{\vartheta}{q}-u\beta-v,\qquad
\frac{\beta\vartheta}{q}=-qw_1,
\]
and
\[
\frac{\vartheta^2}{q}
=v\vartheta-qw_1\beta-q^2uw_1
\]
follow directly from \(F(\beta)=0\).  Consequently
\[
qx^2=r^2\frac{\vartheta}{q}+a
\qquad\text{for some }a\in A.
\]
Since \(x\in T\), the left side belongs to \(T\); since \(r\) is a unit,
this again implies \(\vartheta/q\in T\).  Thus every strict overorder
contains \(S_1=A[\vartheta/q]\).
\end{proof}

Let
\[
A=\Z_q[\beta]\simeq \Z_q[X]/(F(X)),
\qquad
F(X)=X^3+c_2X^2+c_1X+c_0,
\]
where \(c_0,c_1,c_2\in\Z_q\) and \(F\) has nonzero discriminant.

\begin{definition}\label{def:single-residue-square-zero-hnf-sets}
For \(r,s\ge0\), let \(\mathcal S_A(r,s)\) be the set of triples
\[
(U,W,V)\in
(\Z_q/q^r\Z_q)\times(\Z_q/q^s\Z_q)\times(\Z_q/q^{r+s}\Z_q)
\]
satisfying
\begin{align}
V-W^2+c_2W-c_1&\equiv0\pmod {q^s},
\label{eq:general-hnf-1}\\
V+U(U-W)&\equiv0\pmod {q^r},
\label{eq:general-hnf-2}\\
-c_0+c_1U+c_2V-c_2UW-V(U+W)+UW^2
&\equiv0\pmod {q^{r+s}}.
\label{eq:general-hnf-3}
\end{align}
Congruences modulo \(q^0\) impose no condition.
\end{definition}

\begin{proposition}\label{prop:single-residue-square-zero-hnf}
One has
\begin{equation}\label{eq:general-single-residue-hnf-series}
J_A(s)=
\frac{1}{1-t^3}
\sum_{r,s\ge0}\#\mathcal S_A(r,s)t^{r+2s}.
\end{equation}
\end{proposition}

\begin{proof}
Since \(A\) is monogenic, after multiplying the trace dual
by a generator, \(J_A(s)\) counts the ideals of \(A\) by colength.
Use the basis \((1,\beta,\beta^2)\).  Multiplication by \(\beta\) has
matrix
\[
C_F=
\begin{pmatrix}
0&0&-c_0\\
1&0&-c_1\\
0&1&-c_2
\end{pmatrix}.
\]
Write an arbitrary colength-\(n\) lattice in Hermite normal form as
\[
H=
\begin{pmatrix}
q^a&p&r_0\\
0&q^b&s_0\\
0&0&q^c
\end{pmatrix},
\qquad a,b,c\ge0,\quad a+b+c=n.
\]
The condition that the lattice be an ideal is
\begin{equation}\label{eqn: integrality-C_F}
H^{-1}C_FH\in M_3(\Z_q).
\end{equation}
The entries independent of the last column first force
\[
a\ge b\ge c,\qquad p=q^bU,\qquad s_0=q^cW,\qquad r_0=q^cV .
\]
Put
\[
r=a-b,\qquad s=b-c .
\]
Then \(U,W,V\) are residue classes modulo \(q^r,q^s,q^{r+s}\),
respectively.  Substituting the expressions for \(p,s_0,r_0\)
into the remaining three integrality conditions of \eqref{eqn: integrality-C_F} gives exactly
\eqref{eq:general-hnf-1}, \eqref{eq:general-hnf-2}, and
\eqref{eq:general-hnf-3}.  Conversely these congruences make
\(H^{-1}C_FH\) integral, so they parametrize the ideals with fixed
\((r,s,c)\).

For fixed \(r,s\), the parameter \(c\ge0\) is free and contributes
colength
\[
a+b+c=r+2s+3c.
\]
Summing over \(c\) gives the factor \((1-t^3)^{-1}\), proving
\eqref{eq:general-single-residue-hnf-series}.
\end{proof}

For the rest of the section, assume that \(A\) is in the square-zero
first-overorder case of Definition~\ref{def:square-zero-special-fiber}.
Thus
\begin{equation}\label{eq:general-square-zero-polynomial}
F(X)=X^3+quX^2+q^2vX+q^3w,
\qquad u,v,w\in\Z_q .
\end{equation}

\begin{definition}\label{def:general-square-zero-peeled-hnf-sets}
For \(a,b\ge0\),
let \(\mathcal T_A(a,b)\) be the set of triples
\[
(U,W,V)\in
(\Z_q/q^a\Z_q)\times(\Z_q/q^b\Z_q)\times(\Z_q/q^{a+b}\Z_q)
\]
satisfying
\begin{align}
V-q(W^2-uW+v)&\equiv0\pmod {q^b},
\label{eq:general-T-1}\\
V+qU(U-W)&\equiv0\pmod {q^a},
\label{eq:general-T-2}\\
uV-V(U+W)+q(-w+vU-uUW+UW^2)
&\equiv0\pmod {q^{a+b}}.
\label{eq:general-T-3}
\end{align}
\end{definition}

The sets \(\mathcal T_A(a,b)\) have a direct interpretation in the same
Hermite normal form calculation.  For \(a,b,d\ge0\), regard
\(U,W,\widetilde V\in\Z_q\) modulo \(q^a,q^b,q^{a+b+1}\), respectively,
and put
\[
H_{a,b,d}(U,W,\widetilde V):=
q^d
\begin{pmatrix}
q^{a+b+2}&q^{b+2}U&q\widetilde V\\
0&q^{b+1}&qW\\
0&0&1
\end{pmatrix}.
\]
This is the interior HNF matrix in the proof of
Proposition~\ref{prop:single-residue-square-zero-hnf}, with
\(r=a+1\), \(s=b+1\), and with its three off-diagonal parameters written
as \(qU,qW,q\widetilde V\).  Since the scalar \(q^d\) cancels from the
integrality condition, direct substitution gives
\[
H_{a,b,d}(U,W,\widetilde V)^{-1}C_F
H_{a,b,d}(U,W,\widetilde V)\in M_3(\Z_q)
\quad\Longleftrightarrow\quad
(U,W,V)\in\mathcal T_A(a,b),
\]
where \(V\equiv\widetilde V\pmod {q^{a+b}}\).  More precisely, after
division by \(q,q,q^2\), the three remaining integrality conditions are
\eqref{eq:general-T-1}, \eqref{eq:general-T-2}, and
\eqref{eq:general-T-3}, respectively.  Thus \(\mathcal T_A(a,b)\)
parametrizes the normalized interior HNF data.  For fixed \(a,b,d\), each
value of \(V\) modulo \(q^{a+b}\) has exactly \(q\) lifts
\(\widetilde V\) modulo \(q^{a+b+1}\), which accounts for the factor
\(q\) in the following proposition.

\begin{proposition}\label{prop:general-square-zero-first-peeling}
For \(r,s\ge1\), 
\begin{equation}\label{eq:general-square-zero-first-peeling}
\#\mathcal S_A(r,s)=q\,\#\mathcal T_A(r-1,s-1).
\end{equation}
\end{proposition}

\begin{proof}
Using \eqref{eq:general-square-zero-polynomial}, reduction of
\eqref{eq:general-hnf-1}--\eqref{eq:general-hnf-3} modulo \(q\) gives
\[
V=W^2,\qquad V+U(U-W)=0,\qquad -V(U+W)+UW^2=0 .
\]
Substituting \(V=W^2\), the third congruence becomes \(-W^3=0\), hence
\(W=0\).  The second congruence then becomes \(U^2=0\), so \(U=0\), and
the first congruence gives \(V=0\).  Thus these equations force
\[
U\equiv W\equiv V\equiv0\pmod q.
\]
Thus every element of \(\mathcal S_A(r,s)\), with \(r,s\ge1\), can be
written uniquely in the form
\[
U=qU_1,\qquad W=qW_1,\qquad V=qV_1,
\]
where \(U_1,W_1,V_1\) are defined modulo \(q^{r-1},q^{s-1},q^{r+s-1}\).

Substitution into \eqref{eq:general-hnf-1} and
\eqref{eq:general-hnf-2}, followed by division by \(q\), gives
\eqref{eq:general-T-1} and \eqref{eq:general-T-2} with
\((a,b)=(r-1,s-1)\).  Substitution into
\eqref{eq:general-hnf-3} gives
\[
q^2\bigl(
uV_1-V_1(U_1+W_1)+q(-w+vU_1-uU_1W_1+U_1W_1^2)
\bigr),
\]
so after division by \(q^2\) one obtains \eqref{eq:general-T-3} modulo
\(q^{r+s-2}\).

The three resulting congruences depend on \(V_1\) only through its residue
class modulo \(q^{r+s-2}\).  The reduction map
\[
\Z_q/q^{r+s-1}\Z_q\longrightarrow\Z_q/q^{r+s-2}\Z_q
\]
has fibers of cardinality \(q\).  Thus each
\((U_1,W_1,V_1)\in\mathcal T_A(r-1,s-1)\) has exactly \(q\) preimages in
\(\mathcal S_A(r,s)\), obtained by lifting \(V_1\) modulo
\(q^{r+s-1}\).  This proves
\eqref{eq:general-square-zero-first-peeling}.
\end{proof}
Put
\[
G_A^{\partial}(t):=
\sum_{r\ge0}\#\mathcal S_A(r,0)t^r
+\sum_{s\ge1}\#\mathcal S_A(0,s)t^{2s}
\]
and
\[
G_A(t):=
\sum_{a,b\ge0}\#\mathcal T_A(a,b)t^{a+2b}.
\]
Then
\begin{proposition}\label{prop:peeled-hnf-series}
We have 
\begin{equation}\label{eq:peeled-relative-hnf-series}
J_A(s)=
\frac{G_A^{\partial}(t)+qt^3G_A(t)}{1-t^3}.
\end{equation}
\end{proposition}

\begin{proof}
Split the sum in \eqref{eq:general-single-residue-hnf-series} into the
boundary part \(r=0\) or \(s=0\) and the interior part \(r,s\ge1\).  The
boundary part is exactly \(G_A^{\partial}(t)\).  In the interior, apply
Proposition~\ref{prop:general-square-zero-first-peeling} and put
\((a,b)=(r-1,s-1)\); then
\[
t^{r+2s}=t^{a+2b+3}
\]
and the interior contribution becomes \(qt^3G_A(t)\).  This
proves \eqref{eq:peeled-relative-hnf-series}.
\end{proof}

\section{Root counts and the congruence generating functions}
\label{sec:primitive-induction-square-zero}

We express \(G_A^\partial(t)\) and \(G_A(t)\) from
Proposition~\ref{prop:peeled-hnf-series} in terms of a single cubic
polynomial.  For a cubic
\(\Phi\in\Z_q[T]\), put
\[
N_\Phi(k):=
\#\{z\in\Z_q/q^k\Z_q:\Phi(z)\equiv0\pmod {q^k}\},
\qquad k\ge1,
\]
and set \(N_\Phi(0)=1\).

Let \(D_\Phi\in\Z_q[X,Y]\) be the divided difference of \(\Phi\):
\[
\Phi(Y)-\Phi(X)=(Y-X)D_\Phi(X,Y).
\]
Thus \(D_\Phi(X,X)=\Phi'(X)\).

\begin{definition}\label{def:cluster-set}
For \(\epsilon\in\{0,1\}\) and \(a,b\ge\epsilon\), let
\[
m=\min(a,b),\qquad M=\max(a,b),\qquad
P_{\Phi,a,b}(X,Y)=
\begin{cases}
\Phi(Y),&a\le b,\\
\Phi(X),&a>b.
\end{cases}
\]
Let \(\mathfrak C_\Phi^{[\epsilon]}(a,b)\) be the set consisting of the triples
\[
(X,Y,Z)\in
(\Z_q/q^a\Z_q)\times(\Z_q/q^b\Z_q)\times(\Z_q/q^m\Z_q)
\]
satisfying
\begin{align}
D_\Phi(X,Y)&\equiv0\pmod {q^{m-\epsilon}},
\label{eq:cluster-divided-difference}\\
P_{\Phi,a,b}(X,Y)&\equiv0\pmod {q^{M-\epsilon}},
\label{eq:cluster-polynomial}\\
(X-Y)Z&\equiv
\frac{P_{\Phi,a,b}(X,Y)}{q^{M-\epsilon}}
\pmod {q^m}.
\label{eq:cluster-linear}
\end{align}
Congruences modulo \(q^0\) impose no condition.  We write
\(\mathfrak C_\Phi(a,b)=\mathfrak C_\Phi^{[0]}(a,b)\) for \(a,b\ge0\),
and retain the superscript \([1]\) for the shifted family.  The quotient in
the last congruence is formed after choosing lifts of \(X\) and \(Y\) to
\(\Z_q\).  A change of lifts translates \(Z\), so the cardinality, also
after restricting \(X\) and \(Y\) to prescribed residue classes, is
independent of the choices.
\end{definition}

We compare the set \(\mathcal S_A(r,s)\) and \(\mathfrak C_F(r,s)\) for \(A= \Z_q[X]/(F(X))\).
\begin{proposition}\label{prop:scaled-hnf-set-comparison}
Let \(A\), \(F\), and \(\mathcal S_A(r,s)\) be as in
Definition~\ref{def:single-residue-square-zero-hnf-sets}.  For every
\(r,s\ge0\), there is a bijection
\begin{equation}\label{eq:scaled-hnf-set-comparison}
\mathcal S_A(r,s)\simeq \mathfrak C_F(r,s).
\end{equation}
More explicitly, choose lifts \(U,W,V\in\Z_q\) and put
\[
X=-U,\qquad Y=W-c_2.
\]
If \(r\le s\), then \eqref{eq:general-hnf-1} implies that
\[
q^{-s}\bigl(V-W^2+c_2W-c_1\bigr)\in\Z_q.
\]
Let
\[
Z\equiv q^{-s}\bigl(V-W^2+c_2W-c_1\bigr)\pmod{q^r}.
\]
If \(r>s\), then \eqref{eq:general-hnf-2} implies that
\[
q^{-r}\bigl(V+U(U-W)\bigr)\in\Z_q,
\]
and we set
\[
Z\equiv q^{-r}\bigl(V+U(U-W)\bigr)\pmod{q^s}.
\]
These residue classes are well defined under changes of the chosen lifts.These assignments induce the bijection
\eqref{eq:scaled-hnf-set-comparison}.  In particular,
\[
\#\mathcal S_A(r,s)=\#\mathfrak C_F(r,s).
\]
\end{proposition}

\begin{proof}
Denote the left sides of
\eqref{eq:general-hnf-1}--\eqref{eq:general-hnf-3} by \(L,K,M\),
respectively.  Thus
\begin{align*}
L&=V-W^2+c_2W-c_1,\\
K&=V+U(U-W),\\
M&=-c_0+c_1U+c_2V-c_2UW-V(U+W)+UW^2.
\end{align*}
For \(X=-U\) and \(Y=W-c_2\), direct expansion gives
\begin{equation}\label{eq:scaled-hnf-polynomial-identities}
D_F(X,Y)=K-L,\qquad
M=(X-Y)L-F(Y)=(X-Y)K-F(X).
\end{equation}

Suppose first that \(r\le s\).  The condition \(L\equiv0\pmod {q^s}\)
allows us to write \(L=q^sZ\), with \(Z\) defined modulo \(q^r\).
Since \(s\ge r\), the first identity in
\eqref{eq:scaled-hnf-polynomial-identities} shows that
\(K\equiv0\pmod {q^r}\) is equivalent to
\[
D_F(X,Y)\equiv0\pmod {q^r}.
\]
The second identity shows that \(M\equiv0\pmod {q^{r+s}}\) is equivalent
to
\[
F(Y)\equiv0\pmod {q^s},
\qquad
(X-Y)Z\equiv\frac{F(Y)}{q^s}\pmod {q^r}.
\]
These are precisely the defining conditions for
\((X,Y,Z)\in\mathfrak C_F(r,s)\).

If \(r>s\), write \(K=q^rZ\), with \(Z\) defined modulo \(q^s\).
The same identities show that the remaining conditions are equivalent to
\[
D_F(X,Y)\equiv0\pmod {q^s},\qquad
F(X)\equiv0\pmod {q^r},\qquad
(X-Y)Z\equiv\frac{F(X)}{q^r}\pmod {q^s},
\]
which again are precisely the defining conditions for
\(\mathfrak C_F(r,s)\).  In both cases the construction is reversed by
putting \(U=-X\), \(W=Y+c_2\), and
\[
V=
\begin{cases}
W^2-c_2W+c_1+q^sZ,&r\le s,\\
U(W-U)+q^rZ,&r>s.
\end{cases}
\]
After lifts of \(X\) and \(Y\) are fixed, these two constructions are
inverse.  Changing either lift translates \(Z\) as described in
Definition~\ref{def:cluster-set}.  Thus a different choice of lifts only
composes the bijection with a permutation of the \(Z\)-coordinate.  This
proves the result.
\end{proof}

For the rest of the section, return to the square-zero specialization
\eqref{eq:general-square-zero-polynomial}:
\[
F(X)=X^3+quX^2+q^2vX+q^3w.
\]
Define the following auxiliary series:
\begin{equation}\label{eq:auxiliary-hnf-series}
\begin{aligned}
G_\Phi^\partial(t)
&:=1+\sum_{j\ge1}q^{\min(j-1,2)}
N_\Phi(\max(j-3,0))(t^j+t^{2j}),\\
E_\Phi(t)
&:=1+q\sum_{j\ge1}N_\Phi(j-1)(t^j+t^{2j}),\\
\mathcal C_\Phi(t)
&:=\sum_{a,b\ge1}\#\mathfrak C_\Phi^{[1]}(a,b)t^{a+2b}.
\end{aligned}
\end{equation}

Put
\[
\Phi_A(T):=T^3+uT^2+vT+w.
\]
Thus \(F(qT)=q^3\Phi_A(T)\).

\begin{proposition}\label{prop:peeled-hnf-set-comparison}
For every
\(a,b\ge1\), there is a bijection
\begin{equation}\label{eq:peeled-hnf-set-comparison}
\mathcal T_A(a,b)\simeq
\mathfrak C_{\Phi_A}^{[1]}(a,b).
\end{equation}
More explicitly, choose lifts \(U,W,V\in\Z_q\) and put
\[
X=-U,\qquad Y=W-u.
\]
If \(a\le b\), then \eqref{eq:general-T-1} implies that
\[
q^{-b}\bigl(V-q(W^2-uW+v)\bigr)\in\Z_q.
\]
Define
\[
Z\equiv
q^{-b}\bigl(V-q(W^2-uW+v)\bigr)
\pmod{q^a}.
\]
If \(a>b\), then \eqref{eq:general-T-2} implies that
\[
q^{-a}\bigl(V+qU(U-W)\bigr)\in\Z_q,
\]
and define
\[
Z\equiv
q^{-a}\bigl(V+qU(U-W)\bigr)
\pmod{q^b}.
\]
These residue classes are well defined under changes of the chosen lifts.
These assignments induce the bijection
\eqref{eq:peeled-hnf-set-comparison}.  In particular,
\[
\#\mathcal T_A(a,b)
=\#\mathfrak C_{\Phi_A}^{[1]}(a,b).
\]
\end{proposition}

\begin{proof}
Denote the left sides of
\eqref{eq:general-T-1}--\eqref{eq:general-T-3} by \(L,K,M\),
respectively.  Thus
\begin{align*}
L&=V-q(W^2-uW+v),\\
K&=V+qU(U-W),\\
M&=uV-V(U+W)+q(-w+vU-uUW+UW^2).
\end{align*}
For \(X=-U\) and \(Y=W-u\), direct expansion gives
\begin{equation}\label{eq:peeled-hnf-polynomial-identities}
K-L=qD_{\Phi_A}(X,Y),\qquad
M=(X-Y)L-q\Phi_A(Y)=(X-Y)K-q\Phi_A(X).
\end{equation}

If \(a\le b\), write \(L=q^bZ\); if \(a>b\), write \(K=q^aZ\).
Exactly as in the proof of
Proposition~\ref{prop:scaled-hnf-set-comparison},
\eqref{eq:peeled-hnf-polynomial-identities}, after division by \(q\),
turns the remaining two congruences into
\eqref{eq:cluster-divided-difference}--\eqref{eq:cluster-linear} with
\(\epsilon=1\).  Thus the assignments in the statement take values in
\(\mathfrak C_{\Phi_A}^{[1]}(a,b)\).  Conversely, put
\(U=-X\), \(W=Y+u\), and
\[
V=
\begin{cases}
q(W^2-uW+v)+q^bZ,&a\le b,\\
qU(W-U)+q^aZ,&a>b.
\end{cases}
\]
The identities above show that this is the inverse construction.
Changing a lift of \(X\) or \(Y\) only translates \(Z\), as in
Definition~\ref{def:cluster-set}, so the resulting bijection is independent
of the chosen lifts.
\end{proof}
                     
\begin{proposition}
\label{prop:hnf-root-count-parametrization}
The series in Proposition~\ref{prop:peeled-hnf-series} satisfy
\begin{equation}\label{eq:hnf-root-count-parametrization}
G_A^\partial(t)=G_{\Phi_A}^\partial(t),
\qquad
G_A(t)=E_{\Phi_A}(t)+\mathcal C_{\Phi_A}(t).
\end{equation}
\end{proposition}

\begin{proof}
We first count the boundary terms of \(\mathcal S_A(r,s)\).  If \(s=0\),
then \(W=0\), and \eqref{eq:general-hnf-2} eliminates \(V\).
For \(r\ge1\), the remaining congruence
\[
U^3-quU^2+q^2vU-q^3w\equiv0\pmod {q^r}
\]
forces \(U=qZ\) and becomes
\[
q^3\Phi_A(-Z)\equiv0\pmod {q^r},
\]
up to multiplication by \(-1\).  Hence
\[
\#\mathcal S_A(r,0)
=q^{\min(r-1,2)}N_{\Phi_A}(\max(r-3,0)).
\]
If \(r=0\), the same elimination, now using
\eqref{eq:general-hnf-1}, forces \(W=qZ\) and leaves
\(q^3\Phi_A(Z-u)\equiv0\pmod {q^s}\).  Therefore the same formula holds
for \(\#\mathcal S_A(0,s)\), with \(s\) in place of \(r\).  Since
\(\#\mathcal S_A(0,0)=1\), the definition of \(G_A^\partial\) gives the
first identity in \eqref{eq:hnf-root-count-parametrization}.

We next count \(\mathcal T_A(a,b)\).  If \(b=0\), then
\eqref{eq:general-T-2} eliminates \(V\); if \(a=0\), then
\eqref{eq:general-T-1} does.  The remaining congruences are, respectively,
\[
q\Phi_A(-U)\equiv0\pmod {q^a},
\qquad
q\Phi_A(W-u)\equiv0\pmod {q^b},
\]
up to a sign.  Consequently
\[
\#\mathcal T_A(j,0)=\#\mathcal T_A(0,j)
=qN_{\Phi_A}(j-1)\quad (j\ge1),
\qquad
\#\mathcal T_A(0,0)=1.
\]
Their contribution to \(G_A(t)\) is \(E_{\Phi_A}(t)\).  For \(a,b\ge1\),
Proposition~\ref{prop:peeled-hnf-set-comparison} identifies
\(\mathcal T_A(a,b)\) with
\(\mathfrak C_{\Phi_A}^{[1]}(a,b)\).  Hence the interior contribution is
\(\mathcal C_{\Phi_A}(t)\).  This proves the second identity in
\eqref{eq:hnf-root-count-parametrization}.
\end{proof}

\section{Overorder decomposition and recursion}
\label{sec:overorder-recursion}

\subsection{The overorder recursion}

Corollary~\ref{cor:square-zero-finite-induction} classifies the first
index-\(q\) overorders of a cubic order \(S\) with square-zero special
fiber.  The purpose of this section is to extend that first-layer
classification to a finite recursion for the full local overorder series
\(\widetilde L_S(t)\), including the higher-index overorders described in
Remark~\ref{rmk:higher-index-covers}.  This recursion will be compared with
the recursion for Yun's local zeta factor in
Section~\ref{sec:yun-local-zeta-recursion}.

We first treat an arbitrary such order \(S\).  A good basis associates
with \(S\) a binary cubic \(qB\), and each finite projective root
\(\bar z\) of \(\bar B\) determines an index-\(q\) overorder \(S_z\).
These orders are the central intermediate objects of the recurrence: the
overorder series is organized into contributions indexed by the
\(S_z\), and the type of \(S_z\) determines its contribution (Definition \ref{def: root-types}):  simple roots and obstructed multiple roots give terminal terms.
A liftable multiple root gives another order with square-zero special
fiber, represented by a binary cubic \(qB_z^+\), and lowers
\(v_q(\Delta(B))\) by \(2\).  Thus the passage from \(S\) to the orders
\(S_z\) is where the induction step occurs.  We then apply this
construction in Theorem~\ref{thm:q-scaled-overorder-decomposition} to
\[
U=\Z_q[\alpha],
\qquad
S=\Z_q+qU,
\]
and obtain the overorder identity used in
Section~\ref{sec:proof-local-comparison}.  Throughout this section, we use
the local factors \(C_R(t)=h_R(s)\zeta_R(s)\) and the overorder expansion
\eqref{eq:local-overorder-t-series}.

Let \(S\) be a cubic order with square-zero special fiber.  Choose a good
basis \(1,x,y\) of \(S\), so that
\(S=\Z_q\langle1,x,y\rangle\), and let \(X,Y\) be indeterminates.  Suppose
that the binary cubic associated with this basis is
\[
qB(X,Y),
\qquad
B(X,Y)=aX^3+bX^2Y+cXY^2+dY^3\in\Z_q[X,Y].
\]
With the discriminant notation of
Subsection~\ref{subsec: parametrize-cubic}, we have
\(\Delta(qB)=q^4\Delta(B)\ne0\), because \(qB\) is attached to an order
in an \'etale cubic algebra.
The change of good basis \((x,y)\mapsto(-x,-y)\) replaces the associated
form \(qB\) by \(-qB\) without changing \(S\).  Hence \(B\) and \(-B\)
give the same discriminant and projective-root parametrization.

We now apply
Proposition~\ref{prop:universal-square-zero-branching} to \(S\).  In the
chosen good basis, the Delone--Faddeev multiplication table
\eqref{eq:delone-faddeev-multiplication-table}, applied to \(qB\), gives
the following multiplication rules for \(x,y\):
\begin{equation}
x^2=-q^2ac+qbx-qay,
\qquad
xy=-q^2ad,
\qquad
y^2=-q^2bd+qdx-qcy.
\end{equation}
For \(g=rx+sy\), the three quadratic forms in
Proposition~\ref{prop:universal-square-zero-branching} are
\[
\Gamma_0(r,s)=-q(acr^2+2adrs+bds^2),\qquad
\Gamma_1(r,s)=br^2+ds^2,\qquad
\Gamma_2(r,s)=-ar^2-cs^2.
\]
The first congruence in
\eqref{eq:universal-square-zero-branching} is therefore automatic, and
the second is
\[
s\Gamma_1(r,s)-r\Gamma_2(r,s)=B(r,s)\equiv0\pmod q.
\]
Consequently, the projective roots \([\bar r:\bar s]\) of \(\bar B\)
parametrize the index-\(q\) overorders of \(S\), with
\[
S_{[\bar r:\bar s]}:=S+\Z_q\frac{rx+sy}{q}.
\]
For the directed construction below, assume that \([1:0]\) is a root of
\(\bar B(X,Y)\), equivalently \(a\in q\Z_q\).
In particular, for a root \(\bar z\) of \(\bar B(T,1)\), choose a lift
\(z\in\Z_q\) and put
\[
S_\infty:=S+\Z_q\frac{x}{q},
\qquad
S_z:=S+\Z_q\frac{zx+y}{q}.
\]
Set
\[
U:=\Z_q\langle1,x/q,y/q\rangle .
\]

\begin{lemma}
\label{lem:index-q-overorder-joins}
For every root \(\bar z\) of \(\bar B(T,1)\), one has
\[
S_z+S_\infty=U.
\]
Thus \(U\) is the smallest overorder containing both \(S_\infty\) and
\(S_z\).  Moreover,
\[
[U:S]=q^2,
\qquad
[U:S_z]=[U:S_\infty]=q.
\]
\end{lemma}

\begin{proof}
\begin{align*}
S_\infty+S_z
&=S+\Z_q\frac{x}{q}+\Z_q\frac{zx+y}{q}\\
&=\Z_q\langle1,x/q,y/q\rangle .
\end{align*}
Dividing the multiplication relations for \(x,y\) by \(q^2\) shows that the
lattice on the right is closed under multiplication.  It is therefore an
overorder and, being the module sum of \(S_\infty\) and \(S_z\), is the
smallest overorder containing both.  The bases give all three
index formulas.
\end{proof}

For an inclusion \(R\subset R'\) of index \(q\), define the directed
overorder series
\begin{equation}\label{eq:directed-overorder-series}
\mathscr D(R,R';t):=\widetilde L_R(t)-qt^2\widetilde L_{R'}(t).
\end{equation}
It is the contribution of \(R\) and of the overorders of \(R\) that do not
contain the distinguished index-\(q\) overorder \(R'\).

The role of \(S_z\) is more than parametrizing the first overorders of
\(S\).  Lemma~\ref{lem:first-layer-overorder-partition} will show that
every strict overorder \(T\) of \(S\) that does not contain \(S_\infty\)
has
\[
T\cap U=S_z
\]
for a unique finite root \(\bar z\).  Consequently,
\eqref{eq:directed-lattice-decomposition} organizes the corresponding
part of the generating series into the root-indexed contributions
\(qt^2\mathscr D(S_z,U;t)\).  By
Lemma~\ref{lem:directed-index-q-overorder-chart}, the simple and
obstructed cases are terminal; in the liftable case the contribution is
the same square-zero problem for \(B_z^+\), to which induction applies
because of the discriminant drop
\eqref{eq:directed-discriminant-drop}.

\begin{remark}
\label{rem:projective-root-action}
The projective-root parametrization above is compatible with the
Delone--Faddeev action of Section~\ref{subsec: parametrize-cubic}.  Put
\[
e=\begin{pmatrix}x\\y\end{pmatrix}.
\]
For the roots \([1:0]\) and \([\bar z:1]\), respectively, consider
\[
g_\infty=
\begin{pmatrix}q&0\\0&1\end{pmatrix},
\qquad
g_z=
\begin{pmatrix}1&0\\-z&q\end{pmatrix},
\qquad z\in\Z_q,
\]
where \(z\) is any lift of \(\bar z\).  Under the convention of
Section~\ref{subsec: parametrize-cubic}, the basis column
\(e_g:=g^{-1}e\) is associated with the transformed form \((qB)^g\).
Here
\[
g_\infty^{-1}e=
\begin{pmatrix}x/q\\y\end{pmatrix},
\qquad
g_z^{-1}e=
\begin{pmatrix}x\\(zx+y)/q\end{pmatrix}.
\]
The lattices generated by \(1\) and these columns are precisely
\(S_\infty\) and \(S_z\).  The action of
Section~\ref{subsec: parametrize-cubic} gives
\begin{align*}
(qB)^{g_\infty}(X,Y)
&=\frac{a}{q}X^3+bX^2Y+qcXY^2+q^2dY^3,\\
(qB)^{g_z}(X,Y)&=\frac{1}{q}B(qX+zY,Y).
\end{align*}
These forms are integral exactly when \(a\in q\Z_q\) and
\(B(z,1)\in q\Z_q\), respectively.  Since \(B(1,0)=a\), these are
exactly the conditions that \([1:0]\) and \([\bar z:1]\) be roots of
\(\bar B\).  Thus the Delone--Faddeev action recovers the projective-root
parametrization above.
\end{remark}

Since \(S/qS\) has one residue factor and two-dimensional socle, \(S\) is
non-Gorenstein.  Thus
\(h_S(s)=1+qt\) and \(\zeta_S(s)=(1-t)^{-1}\), so
\begin{equation}\label{eq:square-zero-c-series}
C_S(t)=C_\square(t):=\frac{1+qt}{1-t}.
\end{equation}

\begin{lemma}
\label{lem:first-layer-intersection}
Let \(S\) be an order in an \'etale cubic \(\Q_q\)-algebra \(E\), and
suppose that
\[
S=\Z_q\cdot1\oplus M,
\qquad
U=\Z_q\cdot1\oplus q^{-1}M
\]
is an overorder of \(S\), where \(M\) is free of rank two.  If
\(V\supsetneq S\) is an overorder such that \(U\nsubseteq V\), then
\[
[V\cap U:S]=q.
\]
\end{lemma}

\begin{proof}
Let
\[
\pi:E\longrightarrow E/(\Q_q\cdot1),
\qquad
\Lambda_R:=\pi(R)
\]
for every overorder \(R\) of \(S\).  Since the elements of \(R\) are
integral over \(\Z_q\),
\[
R\cap(\Q_q\cdot1)=\Z_q\cdot1.
\]
It follows that, for overorders \(R_1,R_2\) of \(S\),
\[
R_1\subseteq R_2\Longleftrightarrow
\Lambda_{R_1}\subseteq\Lambda_{R_2},
\qquad
\Lambda_{R_1\cap R_2}=\Lambda_{R_1}\cap\Lambda_{R_2},
\]
and, if \(R_1\subseteq R_2\), the map \(\pi\) induces an isomorphism
\[
R_2/R_1\xrightarrow{\sim}\Lambda_{R_2}/\Lambda_{R_1}.
\]

The hypotheses give
\[
\Lambda_U=q^{-1}\Lambda_S,
\qquad
\Lambda_U/\Lambda_S\simeq\mathbb F_q^2.
\]
Choose the least \(j\ge1\) such that
\(q^j\Lambda_V\subseteq\Lambda_S\).  There is
\(\lambda\in\Lambda_V\) such that
\(q^{j-1}\lambda\notin\Lambda_S\), whereas
\(q^j\lambda\in\Lambda_S\).  Hence
\[
q^{j-1}\lambda\in
(\Lambda_V\cap\Lambda_U)\setminus\Lambda_S.
\]
Thus \((\Lambda_V\cap\Lambda_U)/\Lambda_S\) is a nonzero subspace of
\(\Lambda_U/\Lambda_S\).  It is proper because \(U\nsubseteq V\), and
therefore it is one-dimensional.  The preceding intersection and
quotient identities now give
\[
(V\cap U)/S\simeq
(\Lambda_V\cap\Lambda_U)/\Lambda_S,
\]
which proves the assertion.
\end{proof}

\begin{lemma}\label{lem:first-layer-overorder-partition}
Under the hypotheses of Lemma~\ref{lem:first-layer-intersection}, the
strict overorders of \(S\) that do not contain \(U\) are partitioned by
the index-\(q\) overorders \(S'\) satisfying
\[
S\subsetneq S'\subset U.
\]
The part indexed by \(S'\) consists of the overorders that contain \(S'\)
but do not contain \(U\); for each such overorder \(V\), the indexing
overorder is determined uniquely by
\[
S'=V\cap U.
\]
\end{lemma}

\begin{proof}
If \(V\supsetneq S\) and \(U\nsubseteq V\), then
Lemma~\ref{lem:first-layer-intersection} gives
\([V\cap U:S]=q\), so \(V\cap U\) is one of the indicated overorders.
Conversely, if \(S'\subseteq V\) and \(U\nsubseteq V\), the same lemma
shows that \(V\cap U\) has index \(q\) over \(S\).  Since
\(S'\subseteq V\cap U\) and \([S':S]=q\), one has \(V\cap U=S'\).
\end{proof}

\begin{lemma}
\label{lem:directed-lattice-decomposition}
We have 
\begin{equation}\label{eq:directed-lattice-decomposition}
\mathscr D(S,S_\infty;t)
=C_S(t)+qt^2
\sum_{\bar B(\bar z,1)=0}\mathscr D(S_z,U;t).
\end{equation}
\end{lemma}

\begin{proof}
Since \(S_\infty\subseteq U\), an overorder not containing
\(S_\infty\) cannot contain \(U\).  The projective-root parametrization
preceding Lemma~\ref{lem:index-q-overorder-joins}, together with
Lemma~\ref{lem:first-layer-overorder-partition}, partitions the strict
overorders not containing \(S_\infty\) by the finite roots
\(\bar z\) of \(\bar B(T,1)\).  The part indexed by \(\bar z\) consists
of the overorders containing \(S_z\) but not \(U\): the converse follows
from \(S_z+S_\infty=U\).

Since \([S_\infty:S]=q\), the additivity of \(\delta\) gives
\(\delta(T/S)=1+\delta(T/S_\infty)\) for \(T\supseteq S_\infty\).
Using \eqref{eq:local-overorder-t-series}, we therefore obtain
\[
\mathscr D(S,S_\infty;t)
=
\sum_{\substack{S\subseteq T\subseteq\CO_E\\
T\nsupseteq S_\infty}}
q^{\delta(T/S)}t^{2\delta(T/S)}C_T(t).
\]
Separate the term \(T=S\), and partition the remaining terms according
to the unique finite root \(\bar z\) for which \(T\cap U=S_z\).  This
gives
\[
\mathscr D(S,S_\infty;t)
=C_S(t)+
\sum_{\bar B(\bar z,1)=0}
\ \sum_{\substack{S_z\subseteq T\subseteq\CO_E\\
T\nsupseteq U}}
q^{\delta(T/S)}t^{2\delta(T/S)}C_T(t).
\]
For \(T\supseteq S_z\), one has
\(\delta(T/S)=1+\delta(T/S_z)\).  Since
\([U:S_z]=q\) by Lemma~\ref{lem:index-q-overorder-joins}, one also has
\(\delta(T/S_z)=1+\delta(T/U)\) whenever \(T\supseteq U\).  Therefore
the inner sum for \(\bar z\) is
\[
qt^2\bigl(\widetilde L_{S_z}(t)-qt^2\widetilde L_U(t)\bigr)
=qt^2\mathscr D(S_z,U;t).
\]
Summing this identity over the finite roots \(\bar z\) and adding
\(C_S(t)\) proves \eqref{eq:directed-lattice-decomposition}.
\end{proof}

\begin{definition}
\label{def: root-types}
For a root \(\bar z\) of \(\bar B(T,1)\), choose a lift
\(z\in\Z_q\), and write \(B_X=\partial B/\partial X\).  We use the
following three mutually exclusive types:
\begin{align*}
\mathcal R_{\mathrm{simp}}(B)
&:=\{\bar z:\bar B(\bar z,1)=0,\ \overline{B_X}(\bar z,1)\ne0\},\\
\mathcal R_{\mathrm{obs}}(B)
&:=\{\bar z:\bar B(\bar z,1)=\overline{B_X}(\bar z,1)=0,\
B(z,1)\not\equiv0\pmod {q^2}\},\\
\mathcal R_{0}(B)
&:=\{\bar z:\bar B(\bar z,1)=\overline{B_X}(\bar z,1)=0,\
B(z,1)\equiv0\pmod {q^2}\}.
\end{align*} 
We call the elements of \(\mathcal R_{\mathrm{simp}}(B)\),
\(\mathcal R_{\mathrm{obs}}(B)\), and \(\mathcal R_0(B)\) the simple,
obstructed multiple, and liftable multiple roots of \(\bar B(T,1)\),
respectively.
\end{definition}
For a univariate cubic \(P\in\Z_q[T]\), all the preceding notation applied to
\(P\) means the corresponding notation for its homogenization
\(P^{\mathrm h}(X,Y):=Y^3P(X/Y)\).  In particular, if
\(\bar z\in\mathcal R_0(P)\), choose a lift \(z\in\Z_q\) and put
\[
P_z^+(T):=q^{-2}P(z+qT)\in\Z_q[T].
\]
Its homogenization is \((P^{\mathrm h})_z^+\).

These conditions do not depend on the lift of \(\bar z\). For
\(\bar z\in\mathcal R_0(B)\), put
\begin{equation}\label{eq:directed-hensel-rescaling}
B_z^+(X,Y):=q^{-2}B(zY+qX,Y)\in\Z_q[X,Y].
\end{equation}
Explicitly,
\begin{equation}\label{eq:directed-hensel-rescaling-expanded}
B_z^+(T,1)=qaT^3+(3az+b)T^2
+\frac{B_X(z,1)}qT+\frac{B(z,1)}{q^2},
\end{equation}
which also proves the asserted integrality.

\begin{lemma}\label{lem:directed-index-q-overorder-chart}
Let \(\bar z\) be a root of \(\bar B(T,1)\).
\begin{enumerate}[label=\textup{(\roman*)}, leftmargin=2em]
\item If \(\bar z\in\mathcal R_{\mathrm{simp}}(B)\), then \(S_z\) is a
product of \(\Z_q\) and a quadratic order whose residue field is
\(\mathbb F_q\).
\item If \(\bar z\in\mathcal R_{\mathrm{obs}}(B)\), then \(S_z\) is
monogenic with triple-root reduction, and \(U\) is its unique
index-\(q\) overorder.
\item If \(\bar z\in\mathcal R_0(B)\), then \(S_z\) again has square-zero
special fiber.  Put \(\eta_z=(zx+y)/q\).
Then \(1,x,\eta_z\) is a good basis of \(S_z\), its binary cubic form is
\(qB_z^+(X,Y)\), and \(U=S_z+\Z_q(x/q)\).
Applying the preceding projective-root parametrization to \(S_z\), this
index-\(q\) overorder corresponds to \([1:0]\), a root of
\(\overline{B_z^+}(X,Y)\).
\end{enumerate}
\end{lemma}

\begin{proof}
With \(g_z\) as in Remark~\ref{rem:projective-root-action}, the Delone--Faddeev
transformation law and \eqref{eq:directed-hensel-rescaling} give
\[
(qB)^{g_z}(X,Y)=\frac1qB(qX+zY,Y)=qB_z^+(X,Y),
\]
while \(g_z^{-1}(x,y)^{\mathsf t}=(x,\eta_z)^{\mathsf t}\).  This proves
the assertion about the basis and its binary cubic.  Moreover,
\(U=S_z+\Z_q(x/q)\) by
Lemma~\ref{lem:index-q-overorder-joins}.  It remains to determine the
special fiber of \(S_z\).

Replace \(B(X,Y)\) by \(B(X+zY,Y)\); for the transformed form, the chosen
projective root modulo \(q\) is \([0:1]\).  We may therefore assume
\(z=0\) and write
\(B(T,1)=aT^3+bT^2+cT+d\).  Thus \(d=qd_1\) for some
\(d_1\in\Z_q\).  Applying
\eqref{eq:delone-faddeev-multiplication-table} to \(qB(X,Y)\),  the elements
\(1,x,y\)  form a good basis:
\[
x^2=-q^2ac+qbx-qay,
\qquad xy=-q^2ad,
\qquad y^2=-q^2bd+qdx-qcy.
\]
By the definition above,
\(S_0=\Z_q\langle1,x,\eta\rangle\), where \(\eta=y/q\).  Hence
\begin{equation}\label{eq:index-q-overorder-chart-exact-table}
x^2=-q^2ac+qbx-q^2a\eta,
\qquad x\eta=-qad,
\qquad \eta^2=-bd+d_1x-c\eta.
\end{equation}
Modulo \(q\), this becomes
\begin{equation}\label{eq:index-q-overorder-chart-table}
\bar x^2=\bar x\bar\eta=0,
\qquad
\bar\eta^2=\bar d_1\bar x-\bar c\bar\eta.
\end{equation}

If \(\bar c\ne0\), then
\[
e=-\bar c^{-1}\bar\eta+
\bar d_1\bar c^{-2}\bar x
\]
is a nontrivial idempotent.  Idempotents lift over the complete local base,
so \(S_0\) is a product of \(\Z_q\) and a quadratic order.  Moreover
\(e\bar x=0\), so the rank-two factor contains the nonzero
square-zero element \(\bar x\) and has residue field \(\mathbb F_q\).  If
\(\bar c=0\) and \(\bar d_1\ne0\), then
\eqref{eq:index-q-overorder-chart-table} identifies the residual algebra with
\(\mathbb F_q[\bar\eta]/(\bar\eta^3)\), with
\(\bar x=\bar d_1^{-1}\bar\eta^2\).  Hence \(S_0=\Z_q[\eta]\), and the
unique square-zero line in its residual algebra is spanned by \(\bar x\).
The corresponding index-\(q\) overorder is
\[
S_0+\Z_q\frac{x}{q}
=\Z_q\langle1,x/q,y/q\rangle=U.
\]
Finally, if \(\bar c=\bar d_1=0\), then
\eqref{eq:index-q-overorder-chart-table} has square-zero maximal ideal.  These are
exactly the three cases listed.

In case (iii), write \(c=qc_1\) and \(d=q^2d_2\).  Substituting these
identities into \eqref{eq:index-q-overorder-chart-exact-table} and
comparing with \eqref{eq:delone-faddeev-multiplication-table} identifies
that multiplication table with the one attached to the binary cubic
\[
q\left(qaX^3+bX^2Y+c_1XY^2+d_2Y^3\right)
=qB_0^+(X,Y).
\]
The coefficient of \(X^3\) in \(B_0^+\) is divisible by \(q\), so
\([1:0]\) is a root of \(\overline{B_0^+}\).  Under Lemma~\ref{lem:index-q-overorder-joins}, it corresponds to \(S_0+\Z_q(x/q)=U\).
Reversing the change of variables proves all three assertions for \(z\).
\end{proof}

The binary-cubic discriminants are related by 
\begin{equation}\label{eq:directed-discriminant-drop}
\Delta(B_z^+)=q^{-2}\Delta(B).
\end{equation}

Define a power series \(\mathscr D_B(t)\) recursively by
\begin{equation}\label{eq:directed-series-recursion}
\mathscr D_B(t)
=C_\square(t)
+qt^2\frac{\#\mathcal R_{\mathrm{simp}}(B)}{(1-t)^2}
+qt^2\frac{\#\mathcal R_{\mathrm{obs}}(B)}{1-t}
+qt^2\sum_{\bar z\in\mathcal R_0(B)}
\mathscr D_{B_z^+}(t).
\end{equation}
This is a finite recursion: every recursive step lowers
\(v_q(\Delta(B))\) by \(2\), while all cubics occurring in the recursion have
nonzero discriminant.

\begin{proposition}
\label{prop:directed-series-overorders}
In the situation above,
\begin{equation}\label{eq:directed-series-overorders}
\mathscr D_B(t)=\widetilde L_S(t)-qt^2\widetilde L_{S_\infty}(t).
\end{equation}
\end{proposition}

\begin{proof}
We induct on \(v_q(\Delta(B))\).  By
\eqref{eq:square-zero-c-series}, \(C_S=C_\square\).  Hence
Lemma~\ref{lem:directed-lattice-decomposition} reduces the assertion to
evaluating, for each root \(\bar z\in\mathbb F_q\) of
\(\bar B(T,1)\),
\[
\mathscr D(S_z,U;t)
=\widetilde L_{S_z}(t)-qt^2\widetilde L_U(t).
\]

If \(\bar z\) is simple, Lemma
\ref{lem:directed-index-q-overorder-chart} shows that \(S_z\) is a
product of \(\Z_q\) with a quadratic order.  By
\cite[Proposition~3.10]{Malors}, \(U\) is the unique index-\(q\)
overorder of \(S_z\) and is contained in every strict overorder.  Hence
\[
\mathscr D(S_z,U;t)
=C_{S_z}(t)=\frac1{(1-t)^2}.
\]
If \(\bar z\) is multiple with nonzero obstruction, the same lemma makes
\(S_z\) monogenic with one residual factor and identifies \(U\) with
its first overorder.  Proposition~\ref{prop:first-overorder-contained}
shows that every strict overorder contains \(U\).  Therefore
\[
\mathscr D(S_z,U;t)
=C_{S_z}(t)=\frac1{1-t}.
\]

Finally, if \(\bar z\in\mathcal R_0(B)\), Lemma
\ref{lem:directed-index-q-overorder-chart} gives \(S_z\) the binary cubic
\(qB_z^+\) and identifies \(U\) with the index-\(q\) overorder
corresponding to \([1:0]\).  Since
\(v_q(\Delta(B_z^+))=v_q(\Delta(B))-2\), the induction hypothesis gives
\[
\mathscr D(S_z,U;t)=\mathscr D_{B_z^+}(t).
\]
Substitution in
\eqref{eq:directed-lattice-decomposition} gives
\eqref{eq:directed-series-recursion}, and hence
\eqref{eq:directed-series-overorders}.
\end{proof}

For any integral binary cubic \(B(X,Y)\) of nonzero discriminant, define
\begin{equation}\label{eq:branch-sum-definition}
\mathscr B_B(t):=
\frac{\#\mathcal R_{\mathrm{simp}}(B)}{(1-t)^2}
+\frac{\#\mathcal R_{\mathrm{obs}}(B)}{1-t}
+\sum_{\bar z\in\mathcal R_0(B)}\mathscr D_{B_z^+}(t).
\end{equation}

Let
\(\Phi(T)=T^3+uT^2+vT+w\in\Z_q[T]\)
have nonzero discriminant, and set
\[
B_\Phi(X,Y):=X^3+uX^2Y+vXY^2+wY^3.
\]
Let \(E:=\Q_q[T]/(\Phi)\), and let \(\alpha\in E\) be the image of \(T\).
Put
\[
U:=\Z_q[\alpha],
\qquad
S:=\Z_q+qU
=\Z_q\langle1,q\alpha,q\alpha^2\rangle .
\]
Finally, put
\[
x:=-q\alpha,
\qquad
y:=-q(\alpha^2+u\alpha+v),
\qquad
\mathscr B_\Phi(t):=\mathscr B_{B_\Phi}(t).
\]

\begin{theorem}[Overorder decomposition]
\label{thm:q-scaled-overorder-decomposition}
With the preceding notation, \(1,x,y\) is a good basis of \(S\), its
associated binary cubic is \(qB_\Phi\), and
\[
U=\Z_q\langle1,x/q,y/q\rangle .
\]
Moreover, \(S\) has square-zero special fiber in the sense of
Definition~\ref{def:square-zero-special-fiber}, \([U:S]=q^2\), and
\begin{equation}\label{eq:q-scaled-overorder-decomposition}
\widetilde L_S(t)
=C_\square(t)+qt^2\mathscr B_\Phi(t)+q^2t^4\widetilde L_U(t).
\end{equation}
\end{theorem}

\begin{proof}
The elements \(1,x,y\) form a \(\Z_q\)-basis of \(S\), and
\(\Phi(\alpha)=0\) gives
\begin{equation}\label{eq:q-scaled-order-multiplication-table}
x^2=-q^2v+qux-qy,
\qquad
xy=-q^2w,
\qquad
y^2=-q^2uw+qwx-qvy.
\end{equation}
In particular, \(xy\in\Z_q\), so the basis is good.
Comparison with \eqref{eq:delone-faddeev-multiplication-table} shows
that the associated binary cubic is \(qB_\Phi\).  In particular,
\((\bar x,\bar y)^2=0\) in \(S/qS\), so \(S\) has square-zero special
fiber.  The  bases also give
\[
U=\Z_q\langle1,x/q,y/q\rangle,
\qquad [U:S]=q^2.
\]
Since \(B_\Phi(1,0)=1\), the point \([1:0]\) is not a root of
\(\bar B_\Phi\).  Its projective roots are therefore exactly the points
\([\bar z:1]\) for which \(\bar\Phi(\bar z)=0\).

The projective-root parametrization preceding
Lemma~\ref{lem:index-q-overorder-joins} and
Lemma~\ref{lem:first-layer-overorder-partition} therefore show that the
strict overorders not containing \(U\) are the disjoint union, over the
roots \(\bar z\) of \(\bar\Phi\), of the overorders containing
\[
S_z:=S+\Z_q\frac{zx+y}{q}
\]
but not \(U\), where \(z\in\Z_q\) is any lift of \(\bar z\).

The overorders containing \(U\) contribute
\(q^2t^4\widetilde L_U(t)\).  For a fixed \(\bar z\), the overorders
containing \(S_z\) but not \(U\) contribute
\(qt^2\mathscr D(S_z,U;t)\), where the factor \(qt^2\) accounts for
\([S_z:S]=q\).  The proof of
Proposition~\ref{prop:directed-series-overorders} gives
\[
\mathscr D(S_z,U;t)=
\begin{cases}
(1-t)^{-2},&\bar z\in\mathcal R_{\mathrm{simp}}(B_\Phi),\\
(1-t)^{-1},&\bar z\in\mathcal R_{\mathrm{obs}}(B_\Phi),\\
\mathscr D_{(B_\Phi)_z^+}(t),&\bar z\in\mathcal R_0(B_\Phi).
\end{cases}
\]
Therefore, by \eqref{eq:branch-sum-definition}, all strict overorders not
containing \(U\) contribute \(qt^2\mathscr B_\Phi(t)\).  Adding the term
\(C_S=C_\square\) for \(S\) itself proves the formula.
\end{proof}

\section{Recursion for Yun's local zeta factor}
\label{sec:yun-local-zeta-recursion}

We next prove the recursion on Yun's side that matches
Theorem~\ref{thm:q-scaled-overorder-decomposition}.  It expresses the
local zeta factor of the suborder generated by \(q\alpha\) in terms of
that of \(\Z_q[\alpha]\).  We use Yun's local zeta factor \(J_R(s)\)
from \eqref{eq:yun-counting-series}.

Let \(\Phi(T)=T^3+uT^2+vT+w\), let \(U=\Z_q[\alpha]\) with
\(\Phi(\alpha)=0\), and put \(A=\Z_q[q\alpha]\).  Thus
\begin{equation}\label{eq:dilated-polynomial}
A\simeq
\Z_q[X]/\bigl(X^3+quX^2+q^2vX+q^3w\bigr),
\qquad [U:A]=q^3.
\end{equation}

Using the notation \(\mathfrak C_\Phi(a,b)\) from
Definition~\ref{def:cluster-set}, define the congruence generating
function associated with \(\Phi\) by
\begin{equation}\label{eq:full-hnf-series}
\mathcal H_\Phi(t):=
\sum_{a,b\ge0}\#\mathfrak C_\Phi(a,b)t^{a+2b}.
\end{equation}
We use the notation \(\mathcal H_\Phi(t)\) in
\eqref{eq:full-hnf-series} for every cubic \(\Phi\in\Z_q[T]\) of nonzero
discriminant, without requiring \(\Phi\) to be monic.  In particular, if
\(\bar z\in\mathcal R_0(\Phi)\) and \(z\in\Z_q\) is a lift, then
\(\mathcal H_{\Phi_z^+}(t)\) is defined for the integral Hensel transform
\[
\Phi_z^+(T)=q^{-2}\Phi(z+qT)
\]
defined in Section~\ref{sec:overorder-recursion}.

Applying Proposition~\ref{prop:single-residue-square-zero-hnf} to
\(U=\Z_q[\alpha]\) with defining polynomial \(\Phi\), and then using the
bijection of Proposition~\ref{prop:scaled-hnf-set-comparison}, gives
\begin{equation}\label{eq:monogenic-full-hnf-series}
J_U(s)=\frac{\mathcal H_\Phi(t)}{1-t^3}.
\end{equation}

For \(k\ge1\) and \(\bar z\in\mathbb F_q\), put
\[
N_{\Phi,\bar z}(k):=
\#\{x\bmod q^k:x\equiv\bar z\pmod q,\
\Phi(x)\equiv0\pmod {q^k}\}.
\]
Let 
\begin{equation}
\label{def: R-Phi}
R_\Phi(t)=\sum_{k\ge0}N_\Phi(k)t^k,
\qquad
R_{\Phi,\bar z}(t)=\sum_{k\ge1}N_{\Phi,\bar z}(k)t^k.
\end{equation}

\begin{lemma}
\label{lem:root-tower-branch-identity}
For every integral cubic \(\Phi\) of nonzero discriminant,
\begin{equation}\label{eq:root-tower-branch-identity}
R_\Phi(t)-1=t(1-t)\mathscr B_\Phi(t).
\end{equation}
\end{lemma}

\begin{proof}
We argue by induction on \(v_q(\Delta(\Phi))\).  Since every root modulo
\(q^k\) reduces to a root of \(\bar\Phi\),
\begin{equation}\label{eq:root-series-residue-decomposition}
R_\Phi(t)-1=
\sum_{\substack{\bar z\in\mathbb F_q\\ \bar\Phi(\bar z)=0}}
R_{\Phi,\bar z}(t).
\end{equation}
We compare this decomposition with
\eqref{eq:branch-sum-definition}.

If \(\bar z\in\mathcal R_{\mathrm{simp}}(\Phi)\), Hensel's lemma gives
\(N_{\Phi,\bar z}(k)=1\) for every \(k\ge1\).  Hence
\(R_{\Phi,\bar z}(t)=t/(1-t)\).
If \(\bar z\in\mathcal R_{\mathrm{obs}}(\Phi)\), no root modulo \(q^2\)
reduces to \(\bar z\).  Therefore
\(R_{\Phi,\bar z}(t)=t\).

It remains to consider \(\bar z\in\mathcal R_0(\Phi)\).  Choose a lift
\(z\in\Z_q\) .  The substitution
\(x=z+qy\) gives
\[
N_{\Phi,\bar z}(k)=qN_{\Phi_z^+}(k-2)\qquad(k\ge2),
\]
and consequently
\begin{equation}\label{eq:liftable-root-tower}
R_{\Phi,\bar z}(t)=t+qt^2R_{\Phi_z^+}(t).
\end{equation}
Equations \eqref{eq:directed-series-recursion} and
\eqref{eq:branch-sum-definition} give
\[
\mathscr D_{\Phi_z^+}(t)=C_\square(t)+qt^2\mathscr B_{\Phi_z^+}(t).
\]
Moreover,
\(v_q(\Delta(\Phi_z^+))=v_q(\Delta(\Phi))-2\) by
\eqref{eq:directed-discriminant-drop}, so the induction hypothesis
applies to \(\Phi_z^+\).  Using
\(C_\square(t)=(1+qt)/(1-t)\), we obtain
\[
\begin{aligned}
t(1-t)\mathscr D_{\Phi_z^+}(t)
&=t(1-t)C_\square(t)+qt^2\bigl(R_{\Phi_z^+}(t)-1\bigr)\\
&=t(1+qt)+qt^2\bigl(R_{\Phi_z^+}(t)-1\bigr)\\
&=t+qt^2R_{\Phi_z^+}(t)\\
&=R_{\Phi,\bar z}(t).
\end{aligned}
\]
Thus the summand of \eqref{eq:root-series-residue-decomposition}
associated with each root of \(\bar\Phi\) in \(\mathbb F_q\) equals
\(t(1-t)\) times the corresponding summand of
\eqref{eq:branch-sum-definition}.  Summing over these roots proves
\eqref{eq:root-tower-branch-identity}.
\end{proof}

For \(\epsilon\in\{0,1\}\), \(a,b\ge\epsilon+1\), and
\(\bar z\in\mathbb F_q\), put
\[
\mathfrak C_{\Phi,\bar z}^{[\epsilon],\Delta}(a,b)
:=\{(X,Y,Z)\in\mathfrak C_\Phi^{[\epsilon]}(a,b):
X\equiv Y\equiv\bar z\pmod q\}.
\]
We omit the superscript \([0]\).

\begin{lemma}\label{lem:diagonal-hensel-rescaling}
Let \(\epsilon\in\{0,1\}\) and \(a,b\ge\epsilon+1\).  If
\(\mathfrak C_{\Phi,\bar z}^{[\epsilon],\Delta}(a,b)\) is nonempty, then
\(\bar z\in\mathcal R_0(\Phi)\).  Fix a lift \(z\in\Z_q\). 
We have a \(q\)-to-one map
\[
\begin{aligned}
\mathfrak C_{\Phi,\bar z}^{[\epsilon],\Delta}(a,b)
&\longrightarrow
\mathfrak C_{\Phi_z^+}^{[\epsilon]}(a-1,b-1)\\
(X, Y, Z)&\mapsto (X_1,  Y_1, Z_1)
\end{aligned}
\]
with \(X=z+qX_1, Y=z+qY_1, Z_1=Z\bmod{q^{\min(a,b)-1)}}\).
Consequently,
\begin{equation}\label{eq:diagonal-hensel-rescaling}
\#\mathfrak C_{\Phi,\bar z}^{[\epsilon],\Delta}(a,b)
=q\,\#\mathfrak C_{\Phi_z^+}^{[\epsilon]}(a-1,b-1).
\end{equation}
\end{lemma}

\begin{proof}
Put \(m=\min(a,b)\), \(M=\max(a,b)\), and
\[
P=P_{\Phi,a,b}(X,Y),\qquad
Q=\frac{P}{q^{M-\epsilon}}\in\Z_q.
\]
Since \(m-\epsilon\ge1\), congruence
\eqref{eq:cluster-divided-difference} gives
\(\Phi'(z)=D_\Phi(z,z)\equiv0\pmod q\).  Moreover,
\(q\mid X-Y\), so \eqref{eq:cluster-linear} gives \(Q\in q\Z_q\).
If \(W=Y\) for \(a\le b\) and \(W=X\) for \(a>b\), then
\[
\Phi(W)=P\in q^{M-\epsilon+1}\Z_q\subseteq q^2\Z_q.
\]
As \(W\equiv z\pmod q\), Taylor expansion modulo \(q^2\) now gives
\(\Phi(z)\equiv0\pmod {q^2}\).  Hence
\(\bar z\in\mathcal R_0(\Phi)\).

Let
\[
P_1=P_{\Phi_z^+,a-1,b-1}(X_1,Y_1).
\]
Then
\[
D_\Phi(X,Y)=qD_{\Phi_z^+}(X_1,Y_1),
\qquad
P=q^2P_1.
\]
If \(Z_1\) is the reduction of \(Z\) modulo \(q^{m-1}\), the three
conditions in Definition~\ref{def:cluster-set}, after division by \(q\),
become
\[
\begin{gathered}
D_{\Phi_z^+}(X_1,Y_1)\equiv0\pmod {q^{m-1-\epsilon}},\qquad
P_1\equiv0\pmod {q^{M-1-\epsilon}},\\
(X_1-Y_1)Z_1\equiv
\frac{P_1}{q^{M-1-\epsilon}}\pmod {q^{m-1}}.
\end{gathered}
\]
These are precisely the conditions defining
\(\mathfrak C_{\Phi_z^+}^{[\epsilon]}(a-1,b-1)\).

Conversely, \(X_1\) and \(Y_1\) determine \(X\) and \(Y\) uniquely, while
\(Z_1\bmod q^{m-1}\) has exactly \(q\) lifts modulo \(q^m\).
The two identities above show that every lift satisfies the three original
conditions; the last one is unaffected because
\(q(X_1-Y_1)(Z-Z_1)\in q^m\Z_q\).  Thus every fiber has cardinality \(q\).
\end{proof}

Define the nonrecursive part of
\(\mathcal H_\Phi(t)\) by
\begin{align*}
\mathcal H_\Phi^{\mathrm{nr}}(t)
&:=
\sum_{\substack{a,b\ge0\\a=0\text{ or }b=0}}
\#\mathfrak C_\Phi(a,b)t^{a+2b}+
\sum_{a,b\ge1}
\#\Bigl\{(X,Y,Z)\in\mathfrak C_\Phi(a,b):
X\not\equiv Y\pmod q\Bigr\}t^{a+2b}.
\end{align*}

Define the nonrecursive part of \(\mathcal C_\Phi(t)\) by
\begin{align*}
\mathcal C_\Phi^{\mathrm{nr}}(t)
&:=
\sum_{\substack{a,b\ge1\\a=1\text{ or }b=1}}
\#\mathfrak C_\Phi^{[1]}(a,b)t^{a+2b}+
\sum_{a,b\ge2}
\#\Bigl\{(X,Y,Z)\in\mathfrak C_\Phi^{[1]}(a,b):
X\not\equiv Y\pmod q\Bigr\}t^{a+2b}.
\end{align*}

\begin{lemma}
\label{lem:full-hnf-decomposition}
For every cubic \(\Phi\in\Z_q[T]\) of nonzero discriminant, one has
\begin{align}
\mathcal H_\Phi^{\mathrm{nr}}(t)
&=R_\Phi(t)+R_\Phi(t^2)-1
+\sum_{\substack{\bar x,\bar y\in\mathbb F_q\\ \bar x\ne\bar y}}
R_{\Phi,\bar x}(t)R_{\Phi,\bar y}(t^2)
\label{eq:full-hnf-nonrecursive}\\
\mathcal H_\Phi(t)
&=\mathcal H_\Phi^{\mathrm{nr}}(t)
+qt^3\sum_{\bar z\in\mathcal R_0(\Phi)}
\mathcal H_{\Phi_z^+}(t).
\label{eq:full-hnf-recursion}
\end{align}
Moreover,
\begin{equation}\label{eq:neighbor-cluster-recursion}
\mathcal C_\Phi(t)=
\mathcal C_\Phi^{\mathrm{nr}}(t)
+qt^3\sum_{\bar z\in\mathcal R_0(\Phi)}
\mathcal C_{\Phi_z^+}(t).
\end{equation}
\end{lemma}

\begin{proof}
Definition~\ref{def:cluster-set} gives
\[
\#\mathfrak C_\Phi(a,0)=N_\Phi(a),
\qquad
\#\mathfrak C_\Phi(0,b)=N_\Phi(b).
\]
Thus the two boundary components contribute
\(R_\Phi(t)+R_\Phi(t^2)-1\), where the subtraction removes the second
copy of \((a,b)=(0,0)\).

Fix \(a,b\ge1\), and put \(m=\min(a,b)\) and \(M=\max(a,b)\).
If \(X\not\equiv Y\pmod q\), then \(X-Y\) is a unit.  The first two
congruences \eqref{eq:cluster-divided-difference} and
\eqref{eq:cluster-polynomial} specialize to
\[
\begin{array}{ll}
a\le b:&
D_\Phi(X,Y)\equiv0\pmod {q^a},\quad
\Phi(Y)\equiv0\pmod {q^b},\\
a>b:&
D_\Phi(X,Y)\equiv0\pmod {q^b},\quad
\Phi(X)\equiv0\pmod {q^a}.
\end{array}
\]
Because \(Y-X\) is a unit, the identity
\[
\Phi(Y)-\Phi(X)=(Y-X)D_\Phi(X,Y),
\]
shows that, in both cases \(a\le b\) and \(a>b\), the corresponding two
congruences are equivalent to
\[
\Phi(X)\equiv0\pmod {q^a},
\qquad
\Phi(Y)\equiv0\pmod {q^b}.
\]
For each such pair \(X,Y\), congruence \eqref{eq:cluster-linear} determines
\(Z\in\Z_q/q^m\Z_q\) uniquely, since \(X-Y\) is a unit.  Thus, for
fixed distinct classes \(\bar x,\bar y\in\mathbb F_q\), the number of
triples is
\[
N_{\Phi,\bar x}(a)N_{\Phi,\bar y}(b).
\]
Summing over \(a,b\ge1\) and \(\bar x\ne\bar y\) proves
\eqref{eq:full-hnf-nonrecursive}.

For \(\epsilon=0\), the terms of
\(\mathcal H_\Phi-\mathcal H_\Phi^{\mathrm{nr}}\), and for
\(\epsilon=1\), the terms of
\(\mathcal C_\Phi-\mathcal C_\Phi^{\mathrm{nr}}\), are exactly those
with \(a,b\ge\epsilon+1\) and
\[
X\equiv Y\equiv\bar z\pmod q
\]
for a unique \(\bar z\in\mathbb F_q\).
Lemma~\ref{lem:diagonal-hensel-rescaling} shows that only
\(\bar z\in\mathcal R_0(\Phi)\) occur and gives
\[
\#\mathfrak C_{\Phi,\bar z}^{[\epsilon],\Delta}(a,b)
=q\,\#\mathfrak C_{\Phi_z^+}^{[\epsilon]}(a-1,b-1).
\]
Since \(t^{a+2b}=t^3t^{(a-1)+2(b-1)}\), summing this identity for
\(\epsilon=0\) and \(\epsilon=1\) proves
\eqref{eq:full-hnf-recursion} and
\eqref{eq:neighbor-cluster-recursion}, respectively.
\end{proof}

Proposition~\ref{prop:peeled-hnf-series} and
Proposition~\ref{prop:hnf-root-count-parametrization} give
\begin{equation}\label{eq:dilated-hnf-series-compact}
J_A(s)=
\frac{G_\Phi^\partial(t)+qt^3(E_\Phi(t)+\mathcal C_\Phi(t))}
{1-t^3}.
\end{equation}

\begin{proposition}
\label{prop:hnf-neighbor-calculation}
For every cubic \(\Phi\in\Z_q[T]\) of nonzero discriminant,
\begin{equation}\label{eq:hnf-neighbor-calculation}
G_\Phi^\partial(t)+qt^3(E_\Phi(t)+\mathcal C_\Phi(t))
= (1-t^3)
\left(
\frac1{1-t}+qt^2C_\square(t)
+q^2t^4\mathscr B_\Phi(t)
\right)
+q^3t^6\mathcal H_\Phi(t).
\end{equation}
\end{proposition}

To prove Proposition~\ref{prop:hnf-neighbor-calculation}, we first
establish the shifted analogue of
\eqref{eq:full-hnf-nonrecursive}.
Set
\[
S_\Phi(t):=
\sum_{\substack{\bar x,\bar y\in\mathbb F_q\\\bar x\ne\bar y}}
R_{\Phi,\bar x}(t)R_{\Phi,\bar y}(t^2).
\]

\begin{lemma}\label{lem:neighbor-nonrecursive-cluster}
For every cubic \(\Phi\in\Z_q[T]\) of nonzero discriminant,
\begin{equation}\label{eq:neighbor-nonrecursive-cluster}
\begin{aligned}
\mathcal C_\Phi^{\mathrm{nr}}(t)
&=qt^3\left(
(q-1)\bigl(R_\Phi(t)+R_\Phi(t^2)-1\bigr)
+\frac{R_\Phi(t)-1}{t}
+\frac{R_\Phi(t^2)-1}{t^2}-N_\Phi(1)
\right)\\
&\quad+q^2t^3
\sum_{\substack{\bar x,\bar y\in\mathbb F_q\\\bar x\ne\bar y}}
R_{\Phi,\bar x}(t)R_{\Phi,\bar y}(t^2).
\end{aligned}
\end{equation}
\end{lemma}

\begin{proof}
The terms defining \(\mathcal C_\Phi^{\mathrm{nr}}(t)\) form the
disjoint union of
\[
(a,b)=(1,k+1)\quad(k\ge0),\qquad
(a,b)=(k+1,1)\quad(k\ge1),
\]
and the terms with \(a,b\ge2\) and \(X\not\equiv Y\pmod q\).

For \((a,b)=(1,k+1)\) and \(\epsilon=1\), congruence
\eqref{eq:cluster-divided-difference} is vacuous, and the other two
congruences are
\[
\Phi(Y)\equiv0\pmod {q^k},
\qquad
(X-Y)Z\equiv\frac{\Phi(Y)}{q^k}\pmod q;
\]
there are \(qN_\Phi(k)\) possible classes \(Y\bmod q^{k+1}\).
The classes with \(X\not\equiv Y\pmod q\) contribute
\(q(q-1)N_\Phi(k)\), since \(X-Y\) determines \(Z\).  If
\(X\equiv Y\pmod q\), then \eqref{eq:cluster-linear} is equivalent to
\(\Phi(Y)\equiv0\pmod {q^{k+1}}\), giving \(qN_\Phi(k+1)\) triples.
Therefore
\[
\#\mathfrak C_\Phi^{[1]}(1,k+1)
=q((q-1)N_\Phi(k)+N_\Phi(k+1)).
\]
The involution \((X,Y,Z)\mapsto(Y,X,-Z)\) gives the same count for
\(\mathfrak C_\Phi^{[1]}(k+1,1)\).

For the remaining terms, put \(i=a-1\) and \(j=b-1\).
If \(i\le j\), the congruences
\eqref{eq:cluster-divided-difference} and
\eqref{eq:cluster-polynomial} are
\[
D_\Phi(X,Y)\equiv0\pmod {q^i},
\qquad
\Phi(Y)\equiv0\pmod {q^j}.
\]
Since \(X-Y\) is a unit and
\(\Phi(Y)-\Phi(X)=(Y-X)D_\Phi(X,Y)\), these are equivalent to
\[
\Phi(X)\equiv0\pmod {q^i},
\qquad
\Phi(Y)\equiv0\pmod {q^j}.
\]
The case \(i>j\) is the same with \(X\) and \(Y\) interchanged.
Each root modulo \(q^i\), respectively \(q^j\), has \(q\) lifts modulo
\(q^{i+1}\), respectively \(q^{j+1}\).  Thus, for fixed
\(\bar x\ne\bar y\), there are
\(q^2N_{\Phi,\bar x}(i)N_{\Phi,\bar y}(j)\) choices for \(X,Y\), while
\eqref{eq:cluster-linear} determines \(Z\) uniquely.  Consequently,
\[
\mathcal C_\Phi^{\mathrm{nr}}(t)
=qt^3\left(
\sum_{k\ge0}((q-1)N_\Phi(k)+N_\Phi(k+1))t^{2k}
+\sum_{k\ge1}((q-1)N_\Phi(k)+N_\Phi(k+1))t^k
\right)+q^2t^3S_\Phi(t).
\]
Finally, by \eqref{def: R-Phi}
\[
\sum_{k\ge0}((q-1)N_\Phi(k)+N_\Phi(k+1))u^k
=(q-1)R_\Phi(u)+\frac{R_\Phi(u)-1}{u}.
\]
Applying this identity with \(u=t^2\) and \(u=t\), and subtracting
\((q-1)+N_\Phi(1)\) from the latter series, proves
\eqref{eq:neighbor-nonrecursive-cluster}.
\end{proof}

\begin{proof}[Proof of Proposition~\ref{prop:hnf-neighbor-calculation}]
\smallskip\noindent\emph{1. Expansion of the defect.}
From \eqref{eq:auxiliary-hnf-series},
\begin{align}
G_\Phi^\partial(t)
&=1+t+t^2+q(t^2+t^4)
+q^2t^3R_\Phi(t)+q^2t^6R_\Phi(t^2),
\label{eq:neighbor-G-root-series}\\
E_\Phi(t)
&=1+qtR_\Phi(t)+qt^2R_\Phi(t^2).
\label{eq:neighbor-E-root-series}
\end{align}
In \eqref{eq:neighbor-G-root-series}, the terms \(t+t^2\) and
\(q(t^2+t^4)\) arise from \(j=1\) and \(j=2\), respectively, while the
remaining terms are obtained by writing \(k=j-3\) in the sum over
\(j\ge3\).  Equation~\eqref{eq:neighbor-E-root-series} follows by letting
\(k=j-1\).
Let \(\mathfrak E_\Phi(t)\) be the difference between the two sides of
\eqref{eq:hnf-neighbor-calculation}:
\begin{align*}
\mathfrak E_\Phi(t)
&:=G_\Phi^\partial(t)+qt^3\bigl(E_\Phi(t)+\mathcal C_\Phi(t)\bigr)-(1-t^3)\left(
\frac1{1-t}+qt^2C_\square(t)+q^2t^4\mathscr B_\Phi(t)
\right)-q^3t^6\mathcal H_\Phi(t).
\end{align*}

For \(\bar z\in\mathbb F_q\) satisfying \(\bar\Phi(\bar z)=0\), let:
\[
b_{\bar z}(t):=
\begin{cases}
(1-t)^{-2},&\bar z\in\mathcal R_{\mathrm{simp}}(\Phi),\\
(1-t)^{-1},&\bar z\in\mathcal R_{\mathrm{obs}}(\Phi),\\
\mathscr D_{\Phi_z^+}(t),&\bar z\in\mathcal R_0(\Phi).
\end{cases}
\]
Then
\[
\begin{aligned}
R_\Phi(u)
&=1+\sum_{\substack{\bar z\in\mathbb F_q\\
\bar\Phi(\bar z)=0}}R_{\Phi,\bar z}(u),\\
\mathscr B_\Phi(t)
&=\sum_{\substack{\bar z\in\mathbb F_q\\
\bar\Phi(\bar z)=0}}b_{\bar z}(t).
\end{aligned}
\]
Equation \eqref{eq:full-hnf-nonrecursive} becomes
\[
\mathcal H_\Phi^{\mathrm{nr}}(t)
=1+\sum_{\substack{\bar z\in\mathbb F_q\\
\bar\Phi(\bar z)=0}}
\bigl(R_{\Phi,\bar z}(t)+R_{\Phi,\bar z}(t^2)\bigr)+S_\Phi(t).
\]

Since \(N_{\Phi,\bar z}(1)=1\), reindexing gives
\begin{equation}\label{eq:neighbor-root-class-sum}
\begin{aligned}
\sum_{k\ge1}\Bigl((q-1)N_{\Phi,\bar z}(k)
+N_{\Phi,\bar z}(k+1)\Bigr)u^k
&=(q-1)R_{\Phi,\bar z}(u)
+u^{-1}\bigl(R_{\Phi,\bar z}(u)-u\bigr)\\
&=(q-1)R_{\Phi,\bar z}(u)
+u^{-1}R_{\Phi,\bar z}(u)-1.
\end{aligned}
\end{equation}
Applied with \(u=t\) and \(u=t^2\), this identity accounts for the
individual root-class terms in \eqref{eq:neighbor-nonrecursive-cluster}.

We now derive the expansion of \(\mathfrak E_\Phi(t)\) term by term.
Equations \eqref{eq:neighbor-G-root-series} and
\eqref{eq:neighbor-E-root-series} give
\[
\begin{aligned}
G_\Phi^\partial(t)+qt^3E_\Phi(t)
&=1+t+t^2+q(t^2+t^4)+q^2(t^3+t^6)
  +qt^3(1+qt+qt^2)\\
&\quad+q^2t^3\sum_{\substack{\bar z\in\mathbb F_q\\
\bar\Phi(\bar z)=0}}
\Bigl((1+t)R_{\Phi,\bar z}(t)
 +(t^2+t^3)R_{\Phi,\bar z}(t^2)\Bigr).
\end{aligned}
\]
Next, Lemma~\ref{lem:neighbor-nonrecursive-cluster}, together with the
fact that \(N_\Phi(1)\) is the number of roots of \(\bar\Phi\) in
\(\mathbb F_q\), gives
\[
\begin{aligned}
qt^3\mathcal C_\Phi^{\mathrm{nr}}(t)
&=q^2t^6(q-1)+q^3t^6S_\Phi(t)\\
&\quad+q^2t^6(q-1)\sum_{\substack{\bar z\in\mathbb F_q\\
\bar\Phi(\bar z)=0}}
\bigl(R_{\Phi,\bar z}(t)+R_{\Phi,\bar z}(t^2)\bigr)\\
&\quad+q^2t^6\sum_{\substack{\bar z\in\mathbb F_q\\
\bar\Phi(\bar z)=0}}
\bigl(t^{-1}R_{\Phi,\bar z}(t)
+t^{-2}R_{\Phi,\bar z}(t^2)-1\bigr).
\end{aligned}
\]
Similarly, \eqref{eq:full-hnf-nonrecursive} gives
\[
-q^3t^6\mathcal H_\Phi^{\mathrm{nr}}(t)
=-q^3t^6-q^3t^6S_\Phi(t)
-q^3t^6\sum_{\substack{\bar z\in\mathbb F_q\\
\bar\Phi(\bar z)=0}}
\bigl(R_{\Phi,\bar z}(t)+R_{\Phi,\bar z}(t^2)\bigr).
\]
The recursive parts in \eqref{eq:full-hnf-recursion} and
\eqref{eq:neighbor-cluster-recursion} contribute
\[
q^2t^6\sum_{\bar z\in\mathcal R_0(\Phi)}
\mathcal C_{\Phi_z^+}(t)
-q^4t^9\sum_{\bar z\in\mathcal R_0(\Phi)}
\mathcal H_{\Phi_z^+}(t),
\]
whereas \eqref{eq:branch-sum-definition} gives
\[
-(1-t^3)q^2t^4\mathscr B_\Phi(t)
=-q^2t^3\sum_{\substack{\bar z\in\mathbb F_q\\
\bar\Phi(\bar z)=0}}
t(1-t^3)b_{\bar z}(t).
\]
Adding these contributions to the remaining term
\(-(1-t^3)(\frac1{1-t}+qt^2C_\square(t))\) gives
\begin{equation}\label{eq:neighbor-defect-expanded}
\begin{aligned}
\mathfrak E_\Phi(t)
&=\Biggl[
1+t+t^2+q(t^2+t^3+t^4)+q^2(t^3+t^4+t^5)\\
&\qquad-(1-t^3)\left(
\frac1{1-t}+qt^2C_\square(t)\right)
\Biggr]+q^3t^6S_\Phi(t)-q^3t^6S_\Phi(t)\\
&\quad+q^2t^3\sum_{\substack{\bar z\in\mathbb F_q\\
\bar\Phi(\bar z)=0}}
\Biggl[
\begin{aligned}
&(1+t+t^2-t^3)R_{\Phi,\bar z}(t)
+(t+t^2)R_{\Phi,\bar z}(t^2)\\
&\qquad-t^3-t(1-t^3)b_{\bar z}(t)
\end{aligned}
\Biggr]\\
&\quad+\sum_{\bar z\in\mathcal R_0(\Phi)}
\left(
q^2t^6\mathcal C_{\Phi_z^+}(t)
-q^4t^9\mathcal H_{\Phi_z^+}(t)
\right).
\end{aligned}
\end{equation}
Here since \(C_\square(t)=(1+qt)/(1-t)\),
\[
\begin{aligned}
(1-t^3)\left(\frac1{1-t}+qt^2C_\square(t)\right)
&=(1+t+t^2)+qt^2(1+t+t^2)(1+qt)\\
&=1+t+t^2+q(t^2+t^3+t^4)+q^2(t^3+t^4+t^5).
\end{aligned}
\]
Thus the first bracket in \eqref{eq:neighbor-defect-expanded} vanishes,
and the two displayed terms involving \(S_\Phi(t)\) cancel as well.
Consequently, \eqref{eq:neighbor-defect-expanded} has the form
\begin{equation}\label{eq:neighbor-defect-decomposition}
\mathfrak E_\Phi(t)
=\sum_{\substack{\bar z\in\mathbb F_q\\
\bar\Phi(\bar z)=0}}K_{\Phi,\bar z}(t)
+\sum_{\bar z\in\mathcal R_0(\Phi)}
\left(
q^2t^6\mathcal C_{\Phi_z^+}(t)
-q^4t^9\mathcal H_{\Phi_z^+}(t)
\right).
\end{equation}
Here, for \(\bar z\in\mathbb F_q\) satisfying \(\bar\Phi(\bar z)=0\),
\begin{equation}\label{eq:neighbor-single-root-defect}
\begin{aligned}
K_{\Phi,\bar z}(t):=q^2t^3\Bigl(&
(1+t+t^2-t^3)R_{\Phi,\bar z}(t)
+(t+t^2)R_{\Phi,\bar z}(t^2)\\
&-t^3-t(1-t^3)b_{\bar z}(t)\Bigr).
\end{aligned}
\end{equation}

\smallskip\noindent\emph{2. Terminal root types.}
For a simple root and an obstructed multiple root, respectively, one has
\[
\begin{array}{c|cc}
&R_{\Phi,\bar z}(u)&b_{\bar z}(t)\\ \hline
\text{simple}&u/(1-u)&(1-t)^{-2}\\
\text{obstructed}&u&(1-t)^{-1}.
\end{array}
\]
Substitution in \eqref{eq:neighbor-single-root-defect} gives
\(K_{\Phi,\bar z}(t)=0\) in both cases.

\smallskip\noindent\emph{3. Liftable multiple roots and induction.}
Let \(\bar z\in\mathcal R_0(\Phi)\).  Equations
\eqref{eq:liftable-root-tower},
\eqref{eq:directed-series-recursion}, and
\eqref{eq:branch-sum-definition} give
\[
R_{\Phi,\bar z}(u)=u+qu^2R_{\Phi_z^+}(u),
\qquad
b_{\bar z}(t)=C_\square(t)+qt^2\mathscr B_{\Phi_z^+}(t).
\]
Substitution in \eqref{eq:neighbor-single-root-defect}, using
\eqref{eq:neighbor-G-root-series}, \eqref{eq:neighbor-E-root-series}, and
\(C_\square(t)=(1+qt)/(1-t)\), gives
\begin{align*}
K_{\Phi,\bar z}(t)
&=qt^3\Biggl[
G_{\Phi_z^+}^\partial(t)+qt^3E_{\Phi_z^+}(t)-(1-t^3)\left(
\frac1{1-t}+qt^2C_\square(t)+q^2t^4\mathscr B_{\Phi_z^+}(t)
\right)\Biggr]\\
&\quad+q^3t^5(1-t^3)
\bigl(R_{\Phi_z^+}(t)-1-t(1-t)\mathscr B_{\Phi_z^+}(t)\bigr).
\end{align*}
The factor
\(R_{\Phi_z^+}(t)-1-t(1-t)\mathscr B_{\Phi_z^+}(t)\) vanishes by
\eqref{eq:root-tower-branch-identity}.  Hence we obtain
\[
\begin{aligned}
&K_{\Phi,\bar z}(t)
+q^2t^6\mathcal C_{\Phi_z^+}(t)-q^4t^9\mathcal H_{\Phi_z^+}(t)
\\
&\quad=qt^3\Biggl[
G_{\Phi_z^+}^\partial(t)
+qt^3\bigl(E_{\Phi_z^+}(t)+\mathcal C_{\Phi_z^+}(t)\bigr)\\
&\qquad
-(1-t^3)\left(
\frac1{1-t}+qt^2C_\square(t)+q^2t^4\mathscr B_{\Phi_z^+}(t)
\right)
-q^3t^6\mathcal H_{\Phi_z^+}(t)\Biggr]\\
&\quad=qt^3\mathfrak E_{\Phi_z^+}(t),
\end{aligned}
\]
where the last equality is the definition of \(\mathfrak E_{\Phi_z^+}(t)\).
Equation \eqref{eq:neighbor-defect-decomposition} therefore reduces to
\[
\mathfrak E_\Phi(t)=qt^3
\sum_{\bar z\in\mathcal R_0(\Phi)}
\mathfrak E_{\Phi_z^+}(t).
\]
We finish by induction on the nonnegative integer
\(d=v_q(\Delta(\Phi))\).  If \(\mathcal R_0(\Phi)\) is empty, the last
display gives \(\mathfrak E_\Phi(t)=0\).  If it is nonempty, then every
\(\Phi_z^+\) occurring on the right is integral, has nonzero
discriminant, and satisfies
\[
v_q(\Delta(\Phi_z^+))=d-2
\]
by \eqref{eq:directed-discriminant-drop}.  In particular, the latter
valuation is nonnegative, so the cases \(d=0,1\) have
\(\mathcal R_0(\Phi)=\varnothing\) and form the base of the induction.
For \(d\ge2\), the induction hypothesis gives
\(\mathfrak E_{\Phi_z^+}(t)=0\) for every liftable multiple root
\(\bar z\).  The preceding recurrence then gives
\(\mathfrak E_\Phi(t)=0\), which is
\eqref{eq:hnf-neighbor-calculation}.
\end{proof}

\begin{theorem}[Recursion for Yun's local zeta factor]
\label{thm:yun-local-zeta-recursion}
With \(A\), \(U\), and \(\Phi\) as in
\eqref{eq:dilated-polynomial},
\begin{equation}\label{eq:yun-local-zeta-recursion}
J_A(s)=
\frac1{1-t}+qt^2C_\square(t)
+q^2t^4\mathscr B_\Phi(t)
+q^3t^6J_U(s).
\end{equation}
\end{theorem}

\begin{proof}
Divide \eqref{eq:hnf-neighbor-calculation} by \(1-t^3\), and use
\eqref{eq:dilated-hnf-series-compact} and
\eqref{eq:monogenic-full-hnf-series}.
\end{proof}

\section{Proof for Gorenstein cubic orders}
\label{sec:proof-local-comparison}

\begin{proof}[Proof of Theorem~\ref{thm:local-conjecture-a}]
Suppose first that \(q=2\) and \(A\simeq\Z_2^3\).  Then
\(E\simeq\Q_2^3\) and \(A=\CO_E\).  Since
\(\CO_E^\vee=\CO_E\), the lattices counted by \(J_A(s)\) are the
integral ideals of \(\CO_E\), and hence \(J_A(s)=\zeta_E(s)\).
There is no strict overorder and \(h_A(s)=1\), so
\[
\widetilde L_A(t)=\zeta_E(s)=J_A(s).
\]
In every other case, Proposition~\ref{prop:gorenstein-cubic-monogenic}
shows that \(A\) is monogenic.

We now argue by induction on \(\delta(A)\).  Choose a monogenic presentation
\[
A=\Z_q[\theta]\simeq\Z_q[X]/(F_0(X)).
\]
If \(A\) is maximal, Proposition~\ref{prop:local-base-cases} applies, so
assume that \(A\) is nonmaximal.  By
Remark~\ref{rmk:reduction-to-triple-root}, \(\overline F_0\) is then not
squarefree.  Since \(\mathbb F_q\) is perfect, a non-squarefree monic
cubic over \(\mathbb F_q\) has one of the two forms
\[
(X-\bar a)^2(X-\bar b)\quad(\bar a\ne\bar b),
\qquad
(X-\bar a)^3.
\]
In the first case \(\bar b\) is a simple root, and Proposition
\ref{prop:local-base-cases} again applies.  We may therefore assume that
\(\overline F_0=(X-\bar a)^3\).  Choose a lift \(a\in\Z_q\) of \(\bar a\)
and set
\[
\beta=\theta-a,
\qquad
F(X)=F_0(X+a).
\]
Then \(\Z_q[\beta]=\Z_q[\theta]=A\), while
\(\overline F(X)=X^3\).  Hence, for some \(u,v,w\in\Z_q\),
\[
A=\Z_q[\beta],
\qquad
F(\beta)=0,
\qquad
F(X)=X^3+quX^2+qvX+qw.
\]

Since \(A\) is nonmaximal, Proposition
\ref{prop:first-overorder-single-residue} gives \(w=qw_1\).  Corollary
\ref{cor:first-overorder-terminal-cases} handles cases \textup{(a)} and
\textup{(b)} of that proposition.  In the remaining case \textup{(c)},
one has \(v\in q\Z_q\) and \(w_1\in q\Z_q\).  Write
\(v=qv_1\) and \(w_1=qw_2\).  The polynomial above then becomes
\[
F(X)=X^3+quX^2+q^2v_1X+q^3w_2.
\]
Renaming \(v_1,w_2\) as \(v,w\), respectively, we have
\[
F(X)=X^3+quX^2+q^2vX+q^3w,
\qquad
A=\Z_q[\beta],
\qquad
F(\beta)=0.
\]
Set
\[
\alpha=\frac{\beta}{q},
\qquad
\Phi(T)=T^3+uT^2+vT+w.
\]
Since \(F(qT)=q^3\Phi(T)\), one has \(\Phi(\alpha)=0\).  Put
\[
U=\Z_q[\alpha],
\qquad
S=\Z_q+qU.
\]
Then
\[
S=A+\Z_q\frac{\beta^2}{q},
\]
so Proposition~\ref{prop:first-overorder-single-residue} identifies \(S\)
with the unique first overorder of \(A\).  Moreover,
\[
[S:A]=q,\qquad [U:S]=q^2,\qquad [U:A]=q^3,
\]
and therefore \(\delta(U)=\delta(A)-3\).  The order \(U=\Z_q[\alpha]\)
is monogenic, hence Gorenstein.  Since \(\delta(U)<\delta(A)\), the
induction hypothesis gives \(J_U(s)=\widetilde L_U(t)\).  Substituting
this identity into Theorem~\ref{thm:yun-local-zeta-recursion} and then
using the overorder decomposition
\eqref{eq:q-scaled-overorder-decomposition}, we obtain
\[
\begin{aligned}
J_A(s)
&=\frac1{1-t}+qt^2C_\square(t)
+q^2t^4\mathscr B_\Phi(t)
+q^3t^6\widetilde L_U(t)\\
&=\frac1{1-t}+qt^2\bigl(
C_\square(t)+qt^2\mathscr B_\Phi(t)+q^2t^4\widetilde L_U(t)
\bigr)\\
&=\frac1{1-t}+qt^2\widetilde L_S(t).
\end{aligned}
\]
On the other hand, since \(A\) is monogenic and local with one residual factor,
\(C_A(t)=(1-t)^{-1}\).  By Proposition
\ref{prop:first-overorder-contained}, every strict overorder of
\(A\) contains \(S\), and \([S:A]=q\).  Separating the term \(T=A\)
from the remaining terms in \eqref{eq:local-overorder-t-series} therefore
gives
\[
\widetilde L_A(t)
=C_A(t)+qt^2\widetilde L_S(t)
=\frac1{1-t}+qt^2\widetilde L_S(t).
\]
Comparing the two formulas proves
\(J_A(s)=\widetilde L_A(t)\) and completes the induction.
\end{proof}

\begin{corollary}\label{cor:global-conjecture-a}
Let \(R\) be a Gorenstein cubic order over \(\Z\).  Then Conjecture~A
holds:
\[
J_R(s)=L(s,R)\zeta_\Q(s).
\]
\end{corollary}

\begin{proof}
For every prime \(q\), the order
\(R_q:=R\otimes_{\Z}\Z_q\) is Gorenstein.  Apply
Theorem~\ref{thm:local-conjecture-a} to \(R_q\), and use the local-global
products in Section~\ref{sec:local-normalization}.
\end{proof}

\section{Application to the local
\texorpdfstring{\(\GL_3\)}{GL3} Kloosterman series}
\label{sec:kloosterman-application}

We now apply Theorem~\ref{thm:local-conjecture-a} and the HNF
parametrization of Section~\ref{sec:yun-local-zeta-recursion} to the local
factors of the Kloosterman Dirichlet series in
\cite[Definition 19]{DengEspinosaGL3PS}.  The point is that
averaging the HNF congruences over the two variable coefficients  replaces all of the residue-characteristic
and factorization-type case work by one linear map.

Fix a prime \(q\) and \(0\ne c\in\Z_q\), and put
\[
e:=v_q(c).
\]
Since \(c\) is fixed throughout the local calculation, we suppress it
from all indexed notation in this section.
All generating-series identities in this section are initially formal.

\subsection{Kloosterman sums and the HNF reduction}
\label{subsec:kloosterman-sums-hnf-reduction}

Let \(R\) be a cubic \(\Z_q\)-order in an \'etale cubic
\(\Q_q\)-algebra.  Recall that the function \(C_R(t)\) from
\eqref{eq:local-overorder-t-series} may be viewed formally as
\[
C_R(t)=H_R(t)Z_R(t),
\qquad
Z_R(t)=\prod_i(1-t^{f_i})^{-1},
\]
where the product is over the maximal ideals \(\mathfrak m_i\) of \(R\),
\(f_i=[R/\mathfrak m_i:\mathbb F_q]\), and
\[
H_R(t)=
\begin{cases}
1,&R\text{ is Gorenstein},\\
1+qt,&R\text{ is not Gorenstein}.
\end{cases}
\]
Thus \(C_R(t)=h_R(s)\zeta_R(s)\) when \(t=q^{-s}\).  For \(v\ge0\), set
\begin{equation}\label{eq:kloosterman-local-weight}
\kappa_R(v):=[t^v](1-t)C_R(t).
\end{equation}
This is the local coefficient denoted by
\(\Sigma_{R,q}(q^v)\) in \cite[\S8.2]{DengEspinosaGL3PS}.

For \((a,b)\in\Z_q^2\), define
\[
\Phi_{a,b}(X):=X^3-aX^2+bX-c,
\qquad
R(a,b):=\Z_q[X]/(\Phi_{a,b}(X)).
\]
The element \(c\) remains fixed throughout.  For every \(N\ge0\), each
class in \((\Z_q/q^N\Z_q)^2\) has a representative \((a,b)\in\Z_q^2\)
for which \(\Phi_{a,b}\) has nonzero discriminant.  Indeed,
the discriminant is a nonzero polynomial in \((a,b)\), since its
value at \((0,0)\) is \(-27c^2\ne0\), and a nonzero polynomial over
\(\Q_q\) cannot vanish on a \(q\)-adic residue ball.

Following Deng--Espinosa's notation
\cite[\S8.2]{DengEspinosaGL3PS}, with
the constant coefficient \(-c\) fixed, put, for \(v,r\ge0\) and
\((a,b)\) such that
\(\Phi_{a,b}\) has nonzero discriminant,
\[
S_q(a,b;v,r)
:=
\sum_{\substack{R\supseteq R(a,b)\\
[R:R(a,b)]=q^r}}
\kappa_R(v).
\]

\begin{lemma}[Local constancy]\label{lem:kloosterman-local-constancy}
Let \(v,r\ge0\).  Suppose that \((a,b),(a',b')\in\Z_q^2\) satisfy
\[
(a,b)\equiv(a',b')\pmod {q^{v+2r}}
\]
and that both \(\Phi_{a,b}\) and \(\Phi_{a',b'}\) have nonzero
discriminant.  Then
\[
S_q(a,b;v,r)=S_q(a',b';v,r).
\]
\end{lemma}

\begin{proof}
This is 
\cite[Corollary~1]{DengEspinosaGL3PS}.
\end{proof}

For fixed \(v,r\ge0\), choose for each class in
\((\Z_q/q^{v+2r}\Z_q)^2\) a representative
\((a,b)\in\Z_q^2\) for which \(\Phi_{a,b}\) has nonzero discriminant.
Lemma~\ref{lem:kloosterman-local-constancy} makes the following definition
independent of the representatives:
\[
K_q(q^v,q^r)
:=
\sum_{(a,b)\bmod q^{v+2r}}
S_q(a,b;v,r).
\]
Local constancy permits each coefficient class in the
Deng--Espinosa definition to be evaluated at any such representative.
Thus, if \(p\) is a prime and \(k\ge1\), then for \(c=\pm p^k\) the
quantity \(K_q(q^v,q^r)\) is precisely their local Kloosterman sum for
polynomials with constant coefficient \(\mp p^k\).
The local factor of the Deng--Espinosa  Dirichlet series
\cite[Definitions~19]{DengEspinosaGL3PS} is
\begin{equation}\label{eq:kloosterman-local-series}
D_q(z)
:=
\sum_{v,r\ge0}
\frac{K_q(q^v,q^r)}
{q^{v(z+2)+r(2z+3)}}.
\end{equation}
With \(x=q^{-z}\), this is a formal series in \(x\), grouped by the
degree \(v+2r\); every coefficient then contains only finitely many pairs
\((v,r)\).

For \(k\ge0\), set
\[
\mathcal A_k
:=
\sum_{r=0}^{\lfloor k/2\rfloor}
q^{-2k+r}K_q(q^{k-2r},q^r).
\]
\begin{lemma}\label{lem:kloosterman-diagonal-grouping}
With \(x=q^{-z}\), one has the formal power-series identity
\begin{equation}\label{eq:kloosterman-diagonal-series}
D_q(z)=\sum_{k\ge0}\mathcal A_kx^k.
\end{equation}
\end{lemma}

\begin{proof}
For \(v,r\ge0\), set \(k=v+2r\).  Then
\(0\le r\le\lfloor k/2\rfloor\), \(v=k-2r\), and
\[
q^{-v(z+2)-r(2z+3)}
=x^kq^{-2k+r}.
\]
Conversely, every pair \((k,r)\) in this range determines the unique pair
\((v,r)=(k-2r,r)\).  Regrouping
\eqref{eq:kloosterman-local-series} by \(k=v+2r\), the coefficient
of \(x^k\) is therefore
\[
\sum_{r=0}^{\lfloor k/2\rfloor}
q^{-2k+r}K_q(q^{k-2r},q^r)=\mathcal A_k.
\]
This proves \eqref{eq:kloosterman-diagonal-series}.
\end{proof}

For \((a,b)\in\Z_q^2\) such that \(\Phi_{a,b}\) has nonzero
discriminant, write
\[
J_{a,b}(t)
:=
\sum_{n\ge0}b_{a,b}(n)t^n
\]
for Yun's local zeta factor of \(R(a,b)\), viewed as a formal power
series.

\begin{proposition}
\label{prop:kloosterman-yun-identity}
For every \(k\ge0\),
\begin{equation}\label{eq:kloosterman-yun-identity}
\mathcal A_k
=q^{-2k}
\sum_{(a,b)\bmod q^k}
[t^k](1-t)J_{a,b}(t).
\end{equation}
The right-hand side is independent of the representatives chosen for
\((a,b)\bmod q^k\), provided the corresponding cubics have nonzero
discriminant.
\end{proposition}

\begin{proof}
Theorem~\ref{thm:local-conjecture-a} and
\eqref{eq:local-overorder-t-series} give
\[
J_{a,b}(t)
=
\sum_{r\ge0}q^rt^{2r}
\sum_{\substack{R\supseteq R(a,b)\\
[R:R(a,b)]=q^r}}
C_R(t).
\]
The identity first holds for \(t=q^{-s}\) with \(\Re(s)\) large and hence
holds coefficientwise as an identity of power series at \(t=0\).
Multiplying by \(1-t\) and extracting the coefficient of \(t^k\) gives
\[
[t^k](1-t)J_{a,b}(t)
=
\sum_{r=0}^{\lfloor k/2\rfloor}
q^rS_q(a,b;k-2r,r).
\]
Each summand depends only on
\((a,b)\bmod q^k\), by
Lemma~\ref{lem:kloosterman-local-constancy}, because
\((k-2r)+2r=k\).  Summing over the coefficient classes modulo \(q^k\)
proves \eqref{eq:kloosterman-yun-identity} and the asserted
independence of representatives.
\end{proof}

Using the notation of Definition~\ref{def:cluster-set}, with its indices
renamed \(\alpha,\beta\),
equation~\eqref{eq:monogenic-full-hnf-series} specializes to
\begin{equation}\label{eq:kloosterman-hnf-formula}
J_{a,b}(t)
=\frac{1}{1-t^3}
\sum_{\alpha,\beta\ge0}
\#\mathfrak C_{\Phi_{a,b}}(\alpha,\beta)
t^{\alpha+2\beta}.
\end{equation}

\subsection{Averaged congruence counts}

For fixed \(\alpha,\beta\), we average
\(\#\mathfrak C_{\Phi_{a,b}}(\alpha,\beta)\), the number of triples
satisfying the congruences in Definition~\ref{def:cluster-set}, over
the coefficient classes
\((a,b)\bmod q^{\alpha+\beta}\).

For \(d\ge0\) and \(0\le j\le d\), define
\[
\omega_d(j):=
\begin{cases}
1-q^{-1},&0\le j<d,\\
1,&j=d.
\end{cases}
\]

\begin{lemma}
\label{lem:kloosterman-averaged-cluster}
For \(\alpha,\beta\ge0\), put \(n=\alpha+\beta\).  Then
\begin{equation}\label{eq:kloosterman-averaged-cluster}
q^{-2n}
\sum_{(a,b)\bmod q^n}
\#\mathfrak C_{\Phi_{a,b}}(\alpha,\beta)
=
\sum_{\substack{0\le\rho\le\alpha,\ 0\le\sigma\le\beta\\
\rho+\sigma\le e}}
\omega_\alpha(\rho)\omega_\beta(\sigma).
\end{equation}
\end{lemma}

\begin{proof}
For \(d\ge0\) and \(\bar X\in\Z_q/q^d\Z_q\), choose a lift
\(X\in\Z_q\) and define
\[
\nu_d(\bar X):=\min\{v_q(X),d\},
\]
where \(v_q(0)=+\infty\).  This is independent of the lift; in particular,
\(\nu_d(0)=d\).  Fix classes
\(\bar X\bmod q^\alpha\) and \(\bar Y\bmod q^\beta\), and put
\[
\rho=\nu_\alpha(\bar X),
\qquad
\sigma=\nu_\beta(\bar Y).
\]
Choose representatives \(X,Y\in\Z_q\) of \(\bar X,\bar Y\) with
\[
v_q(X)=\rho,\qquad v_q(Y)=\sigma;
\]
for the zero classes one may take \(X=q^\alpha\) and \(Y=q^\beta\),
respectively.  The case \(n=0\) is immediate, so suppose \(n>0\).

If \(\alpha\ne\beta\), the involution
\[
(\bar X,\bar Y,Z)\longmapsto(\bar Y,\bar X,-Z)
\]
identifies \(\mathfrak C_\Phi(\alpha,\beta)\) with
\(\mathfrak C_\Phi(\beta,\alpha)\).  The case \(\alpha=\beta\) requires
no interchange.  The right-hand side of
\eqref{eq:kloosterman-averaged-cluster} is also symmetric in
\(\alpha,\beta\).  It therefore suffices to assume
\(\alpha\le\beta\).  By
\eqref{eq:cluster-divided-difference}--\eqref{eq:cluster-linear}, the
three conditions, relative to the chosen representatives \(X,Y\), are
equivalent to
\[
D_{\Phi_{a,b}}(X,Y)\equiv0\pmod {q^\alpha},
\qquad
\Phi_{a,b}(Y)
\equiv q^\beta(X-Y)Z\pmod {q^n}.
\]
Recall that 
\[
D_{\Phi_{a,b}}(X,Y)
=X^2+XY+Y^2-a(X+Y)+b,
\]
the first congruence is equivalent to
\[
b\equiv a(X+Y)-(X^2+XY+Y^2)\pmod {q^\alpha}.
\]
As \(n=\alpha+\beta\), the classes \(b\bmod q^n\) satisfying this
congruence are uniquely parametrized by \(\tau\bmod q^\beta\) through
\[
b\equiv
a(X+Y)-(X^2+XY+Y^2)+q^\alpha\tau\pmod {q^n}.
\]
Put
\[
\eta:=a-X-Y\in\Z_q/q^n\Z_q.
\]
Substitution modulo \(q^n\) gives
\[
\Phi_{a,b}(Y)\equiv XY\eta+q^\alpha Y\tau-c\pmod {q^n}.
\]
Define the homomorphism
\[
\mathcal L_{X,Y}\colon
(\Z_q/q^n\Z_q)\times(\Z_q/q^\beta\Z_q)
\times(\Z_q/q^\alpha\Z_q)
\longrightarrow \Z_q/q^n\Z_q
\]
by
\[
\mathcal L_{X,Y}(\eta,\tau,Z)
:=XY\eta+q^\alpha Y\tau-q^\beta(X-Y)Z.
\]
For the fixed representatives \(X,Y\), this is well defined: changing
\(\eta,\tau,Z\) by \(q^n,q^\beta,q^\alpha\), respectively, changes its
value by a multiple of \(q^n\).  Thus the triples \((a,b,Z)\) associated
with the fixed classes \(\bar X,\bar Y\) are in bijection with the
solutions of
\begin{equation}\label{eq:kloosterman-linear-cluster}
\mathcal L_{X,Y}(\eta,\tau,Z)=c
\quad\text{in }\Z_q/q^n\Z_q,
\end{equation}
where
\[
\eta\in \Z_q/q^n\Z_q,
\qquad
\tau\in \Z_q/q^\beta\Z_q,
\qquad
Z\in \Z_q/q^\alpha\Z_q.
\]
By construction, \(v_q(XY)=\rho+\sigma\).  Hence the restriction
of \(\mathcal L_{X,Y}\) to the triples \((\eta,0,0)\) has image
\[
q^{\rho+\sigma}\Z_q/q^n\Z_q.
\]
 We observe that 
\[
\alpha+\sigma\ge\rho+\sigma
\]
and, since \(\rho\le\alpha\le\beta\),
\[
\beta+v_q(X-Y)
\ge\beta+\min(\rho,\sigma)
\ge\rho+\sigma.
\]
Hence the coefficient \(q^\alpha Y\) of \(\tau\) and the coefficient
\(-q^\beta(X-Y)\) of \(Z\) both have valuation at least
\(\rho+\sigma\). 

Let us denote by \(\operatorname{im}(\mathcal L_{X,Y})\)
the image of \(\mathcal L_{X,Y}\), then
\[
\operatorname{im}(\mathcal L_{X,Y})
=q^{\rho+\sigma}\Z_q/q^n\Z_q.
\]
Thus the homomorphism itself depends on the chosen representatives
\(X,Y\), but its image depends only on \(\rho,\sigma\).  In particular,
\eqref{eq:kloosterman-linear-cluster} is soluble precisely when
\[
\rho+\sigma\le v_q(c)=e.
\]
Under this condition the image has cardinality
\(q^{n-\rho-\sigma}\).  Since the domain has cardinality \(q^{2n}\),
the fiber \(\mathcal L_{X,Y}^{-1}(c)\) has cardinality
\(q^{n+\rho+\sigma}\).

It remains to sum over \(\bar X,\bar Y\).  The number of classes
\(\bar X\bmod q^\alpha\) satisfying \(\nu_\alpha(\bar X)=\rho\) is
\[
\#\{\bar X\bmod q^\alpha:\nu_\alpha(\bar X)=\rho\}
=
\begin{cases}
q^{\alpha-\rho}(1-q^{-1}),&\rho<\alpha,\\
1,&\rho=\alpha.
\end{cases}
\]
Hence
\begin{equation}
\label{eqn: counting-identity}
q^{-\alpha}
\#\{\bar X\bmod q^\alpha:\nu_\alpha(\bar X)=\rho\}q^\rho
=\omega_\alpha(\rho),
\end{equation}
and the analogous identity holds for \(\bar Y\).  For each fixed
\(\bar X,\bar Y\), the preceding fiber calculation counts the triples
\((a,b,Z)\)
satisfying the conditions of Definition~\ref{def:cluster-set}.
Therefore
\begin{equation}
q^{-2n}
\sum_{(a,b)\bmod q^n}
\#\mathfrak C_{\Phi_{a,b}}(\alpha,\beta)
=
q^{-2n}
\sum_{\substack{\bar X\bmod q^\alpha,\ \bar Y\bmod q^\beta\\
\rho+\sigma\le e}}
q^{n+\rho+\sigma}.
\end{equation}
Since \(n=\alpha+\beta\), we have 
$
q^{-2n}q^{n+\rho+\sigma}
=(q^{-\alpha}q^\rho)(q^{-\beta}q^\sigma).
$
Grouping the classes \(\bar X,\bar Y\) according to \(\rho,\sigma\) and applying
\eqref{eqn: counting-identity} to each factor now gives
\[
q^{-2n}
\sum_{(a,b)\bmod q^n}
\#\mathfrak C_{\Phi_{a,b}}(\alpha,\beta)
=
\sum_{\substack{0\le\rho\le\alpha,\ 0\le\sigma\le\beta\\
\rho+\sigma\le e}}
\omega_\alpha(\rho)\omega_\beta(\sigma),
\]
which is \eqref{eq:kloosterman-averaged-cluster}.
\end{proof}

\subsection{Evaluation of the local factors}

\begin{proof}[Proof of Theorem~\ref{thm:kloosterman-local-factor}]
Define
\[
B_k
:=
q^{-2k}
\sum_{(a,b)\bmod q^k}
[t^k]J_{a,b}(t),
\qquad
B_{-1}:=0.
\]
The HNF formula \eqref{eq:kloosterman-hnf-formula} shows that
\([t^j]J_{a,b}(t)\) depends only on
\((a,b)\bmod q^j\).  Hence, for \(k\ge1\),
\[
\begin{aligned}
q^{-2k}
\sum_{(a,b)\bmod q^k}
[t^{k-1}]J_{a,b}(t)
&=
q^{-2k}q^2
\sum_{(a,b)\bmod q^{k-1}}
[t^{k-1}]J_{a,b}(t)\\
&=B_{k-1}.
\end{aligned}
\]
Together with the case \(k=0\),
Proposition~\ref{prop:kloosterman-yun-identity} gives
\begin{equation}\label{eq:kloosterman-A-B-relation}
\mathcal A_k=B_k-B_{k-1}
\qquad(k\ge0).
\end{equation}

From \eqref{eq:kloosterman-hnf-formula} we obtain 
\[
[t^k]J_{a,b}(t)
=
\sum_{\substack{\alpha,\beta,h\ge0\\
\alpha+2\beta+3h=k}}
\#\mathfrak C_{\Phi_{a,b}}(\alpha,\beta).
\]

 It follows from definition that
\(\#\mathfrak C_{\Phi_{a,b}}(\alpha,\beta)\) depends only on
\((a,b)\bmod q^{\alpha+\beta}\).

If \(\alpha+2\beta+3h=k\), then
\(k-(\alpha+\beta)=\beta+3h\ge0\).  Each pair modulo
\(q^{\alpha+\beta}\) therefore has
\(q^{2(k-\alpha-\beta)}\) lifts modulo \(q^k\).  Substituting the
coefficient formula above into the definition of \(B_k\) and
interchanging the finite sums gives
\[
\begin{aligned}
B_k
&=
\sum_{\substack{\alpha,\beta,h\ge0\\
\alpha+2\beta+3h=k}}
q^{-2k}
\sum_{(a,b)\bmod q^k}
\#\mathfrak C_{\Phi_{a,b}}(\alpha,\beta)\\
&=
\sum_{\substack{\alpha,\beta,h\ge0\\
\alpha+2\beta+3h=k}}
q^{-2k}q^{2(k-\alpha-\beta)}
\sum_{(a,b)\bmod q^{\alpha+\beta}}
\#\mathfrak C_{\Phi_{a,b}}(\alpha,\beta)\\
&=
\sum_{\substack{\alpha,\beta,h\ge0\\
\alpha+2\beta+3h=k}}
q^{-2(\alpha+\beta)}
\sum_{(a,b)\bmod q^{\alpha+\beta}}
\#\mathfrak C_{\Phi_{a,b}}(\alpha,\beta).
\end{aligned}
\]
Lemma~\ref{lem:kloosterman-averaged-cluster} now gives
\[
\begin{aligned}
\sum_{k\ge0}B_kx^k
&=
\frac1{1-x^3}
\sum_{\alpha,\beta\ge0}
\sum_{\substack{0\le\rho\le\alpha,\ 0\le\sigma\le\beta\\
\rho+\sigma\le e}}
\omega_\alpha(\rho)\omega_\beta(\sigma)x^{\alpha+2\beta}\\
&=
\frac1{1-x^3}
\sum_{\substack{\rho,\sigma\ge0\\\rho+\sigma\le e}}
\left(\sum_{\alpha\ge\rho}\omega_\alpha(\rho)x^\alpha\right)
\left(\sum_{\beta\ge\sigma}\omega_\beta(\sigma)x^{2\beta}\right).
\end{aligned}
\]
For fixed \(\rho,\sigma\),
\[
\sum_{\alpha\ge\rho}\omega_\alpha(\rho)x^\alpha
=x^\rho\frac{1-q^{-1}x}{1-x},
\qquad
\sum_{\beta\ge\sigma}\omega_\beta(\sigma)x^{2\beta}
=x^{2\sigma}\frac{1-q^{-1}x^2}{1-x^2}.
\]
It follows that
\[
\sum_{k\ge0}B_kx^k
=
\frac{(1-q^{-1}x)(1-q^{-1}x^2)}
{(1-x)(1-x^2)(1-x^3)}
\sum_{\substack{\rho,\sigma\ge0\\\rho+\sigma\le e}}
x^{\rho+2\sigma}.
\]
By \eqref{eq:kloosterman-diagonal-series} and
\eqref{eq:kloosterman-A-B-relation},
\[
D_q(z)
=
(1-x)\sum_{k\ge0}B_kx^k,
\]
and hence
\[
D_q(z)
=
\frac{(1-q^{-1}x)(1-q^{-1}x^2)}
{(1-x^2)(1-x^3)}
\sum_{\substack{\rho,\sigma\ge0\\\rho+\sigma\le e}}
x^{\rho+2\sigma}.
\]
Finally,
\[
\sum_{\substack{\rho,\sigma\ge0\\\rho+\sigma\le e}}
x^{\rho+2\sigma}
=
\frac{(1-x^{e+1})(1-x^{e+2})}
{(1-x)(1-x^2)},
\]
and \eqref{eq:kloosterman-local-factor-rational} follows.
\end{proof}

The grouped formal power series in
Theorem~\ref{thm:kloosterman-local-factor} converges for \(|x|<1\), or
\(\Re(z)>0\), and the rational formula gives its meromorphic
continuation.  For \(e=0\), it specializes to
\begin{equation}\label{eq:kloosterman-unit-local-factor}
D_q(z)
=
\frac{(1-q^{-z-1})(1-q^{-2z-1})}
{(1-q^{-2z})(1-q^{-3z})}.
\end{equation}

\begin{corollary}
\label{cor:kloosterman-prime-power}
Let \(p\) be a prime, let \(k\ge1\), and take 
\(c=p^k\), so the polynomial has constant coefficient \(-p^k\).
At every local prime \(q\), let \(D_q(z)\) be the series
\eqref{eq:kloosterman-local-series} formed with the image of
\(p^k\) in \(\Z_q\).  Then
\[
D_q(z)=
\begin{cases}
\displaystyle
\frac{(1-q^{-z-1})(1-q^{-2z-1})}
{(1-q^{-2z})(1-q^{-3z})},
&q\ne p,\\[14pt]
\displaystyle
\frac{
(1-p^{-z(k+1)})(1-p^{-z(k+2)})
(1-p^{-z-1})(1-p^{-2z-1})
}{
(1-p^{-z})(1-p^{-2z})^2(1-p^{-3z})
},
&q=p.
\end{cases}
\]
For \(k=1\), the distinguished factor simplifies to
\[
D_p(z)
=
\frac{(1-p^{-z-1})(1-p^{-2z-1})}
{(1-p^{-z})(1-p^{-2z})}.
\]
\end{corollary}

\begin{proof}
Apply Theorem~\ref{thm:kloosterman-local-factor} to the image of \(p^k\)
in \(\Z_q\).  Its valuation is zero for \(q\ne p\) and \(k\) for
\(q=p\).  When \(k=1\), the factors \(1-x^2\) and \(1-x^3\) cancel in
\eqref{eq:kloosterman-local-factor-rational}.
\end{proof}

Replacing \(c=p^k\) by \(c=-p^k\) changes the polynomial's constant
coefficient from \(-p^k\) to \(+p^k\), but it does not change any local
factor, by Theorem~\ref{thm:kloosterman-local-factor}.  Thus
Corollary~\ref{cor:kloosterman-prime-power} covers both signs in the
Deng--Espinosa series.

\begin{corollary}\label{cor:kloosterman-euler-product}
Let \(0\ne c\in\Z\).  For every prime \(q\), let \(D_q(z)\) denote the
series \eqref{eq:kloosterman-local-series}
formed with the image of \(c\) in \(\Z_q\).  For each prime
\(p\mid c\), write \(k_p=v_p(c)\).  Then, for \(\Re(z)>1\),
\[
\prod_q D_q(z)
=
\frac{\zeta(2z)\zeta(3z)}
{\zeta(z+1)\zeta(2z+1)}
\prod_{p\mid c}
\frac{(1-p^{-z(k_p+1)})(1-p^{-z(k_p+2)})}
{(1-p^{-z})(1-p^{-2z})}.
\]
When \(|c|=1\), the finite product over \(p\mid c\) is understood to be
\(1\).
The right-hand side gives the meromorphic continuation of the product.
\end{corollary}

\begin{proof}
The product of the unit factors is
\[
\prod_q
\frac{(1-q^{-z-1})(1-q^{-2z-1})}
{(1-q^{-2z})(1-q^{-3z})}
=
\frac{\zeta(2z)\zeta(3z)}
{\zeta(z+1)\zeta(2z+1)}.
\]
If \(p\mid c\), then \(v_p(c)=k_p\), and
Theorem~\ref{thm:kloosterman-local-factor} replaces the unit factor at
\(p\) by its product with
\[
\frac{(1-p^{-z(k_p+1)})(1-p^{-z(k_p+2)})}
{(1-p^{-z})(1-p^{-2z})}.
\]
For \(q\nmid c\), one has \(v_q(c)=0\), so no correction is present.
Multiplying the finitely many corrections over \(p\mid c\) proves the
formula.
\end{proof}

\bibliographystyle{plain}
\bibliography{Bibliography}

\end{document}